\documentclass[12pt,authoryear,letterpaper,reqno,oneside]{amsart}
\pdfoutput=1
\usepackage[left=1.45in,right=1.45in,bottom=1.40in,top=1.40in]{geometry}
\usepackage{amsmath,amsfonts,amssymb}%
\usepackage[usenames]{color}
\usepackage{setspace}
\usepackage{comment}
\usepackage{stmaryrd}
\usepackage{enumerate}
\usepackage{microtype}
\usepackage{nicefrac}

\definecolor{darkred}{rgb}{0.5,0,0}
\definecolor{darkgreen}{rgb}{0, 0.3,0}
\definecolor{darkblue}{rgb}{0,0,0.6}
\definecolor{LightGray}{rgb}{.6,.6,.6}

\usepackage[colorlinks,citecolor=darkblue,urlcolor=darkblue,linkcolor=darkgreen]{hyperref}

\theoremstyle{plain}
\newtheorem{proposition}{Proposition}[section]
\newtheorem{lemma}[proposition]{Lemma}
\newtheorem{corollary}[proposition]{Corollary}

\newtheorem{theorem}[proposition]{Theorem}

\theoremstyle{definition}
\newtheorem{definition}[proposition]{Definition}

\theoremstyle{remark}
\newtheorem{remark}[proposition]{Remark}

\newtheorem{example}[proposition]{Example}

\newcommand{\enum}[1]{|\!|{#1}|\!|}

\DeclareMathOperator*{\Aut}{Aut}

\DeclareMathOperator{\ar}{ar}
\DeclareMathOperator{\Bernoulli}{Bernoulli}
\newcommand{\tdef}{\textrm{def}}
\newcommand{\qf}{\textrm{qf}}
\newcommand{\tth}{\textrm{th}}
\newcommand{\fin}{\textrm{fin}}
\DeclareMathOperator{\Sp}{Sp}
\DeclareMathOperator{\Th}{Th}
\DeclareMathOperator{\qftp}{qftp}
\DeclareMathOperator{\tp}{tp}

\DeclareMathOperator{\dcl}{dcl}

\newcommand{\doublewedge}{\bigwedge}
\newcommand{\doublevee}{\bigvee}
\newcommand{\bigdoublewedge}{\bigwedge}
\newcommand{\bigdoublevee}{\bigvee}

\newcommand{\cc}{{\mathbf{c}}}

\newcommand{\n}{\mathbb{N}}
\newcommand{\bbM}{\mathbb{M}}

\renewcommand{\aa}{\overline{a}}
\newcommand{\bb}{\overline{b}}
\renewcommand{\cc}{\overline{c}}
\newcommand{\xx}{\overline{x}}
\newcommand{\yy}{\overline{y}}

\def\fM{{\EM{\mathfrak{M}}}}

\def\Th{{\text{\textnormal{Th}}}}

\newcommand{\onehalf}{\nicefrac{1}{2}}

\newcommand{\sympar}[1]{\ensuremath{\mathrm{Sym}({#1})}}
\newcommand{\sym}{\sympar{\Nats}}

\def\Str{{\mathrm{Str}}}

\def\extent#1{\EM{\llbracket{#1}\rrbracket}}
\def\bextent#1{\EM{\bigl\llbracket{#1}\bigr\rrbracket}}
\def\Bextent#1{\EM{\Bigl\llbracket{#1}\Bigr\rrbracket}}

\def\Lomega#1{{\EM{\mc{L}_{#1, \w}}}}

\def\Lwow{\Lomega{\w_1}}

\newcommand{\modelsas}{\models}

\newcommand{\FO}{{\mathrm{FO}}}

\newcommand{\Nats}{\Naturals}

\def\w{\EM{\omega}}
\def\Naturals{{\EM{{\mbb{N}}}}}

\def\^{\EM{{}^{\And}}}

\def\And{\EM{\wedge}}

\def\<{\EM{\langle}}
\def\>{\EM{\rangle}}

\newcommand{\defn}[1]{\textbf{#1}}
\newcommand{\defas}{:=}

\def\EM#1{\ensuremath{#1}}

\def\mbb#1{\EM{\mathbb{#1}}}

\def\mc#1{\EM{\mathcal{#1}}}

\newcommand{\ER}{Erd\H{o}s--R\'{e}nyi}

\definecolor{MyGreen}{rgb}{.75,0,.75}
\definecolor{RealGreen}{rgb}{0,1,0}
\definecolor{DarkGreen}{rgb}{0.1,0.4,0.1}
\definecolor{ActualGreen}{rgb}{0.0,0.5,0.0}
\definecolor{MyBlue}{rgb}{0,0,1}
\definecolor{MyRed}{rgb}{1,0,0}
\definecolor{DarkRed}{rgb}{.6,0,0}

\begin{document}


\title[Properly ergodic structures]
{Properly ergodic structures}

\author[Ackerman]{Nathanael Ackerman}
\address{
Department of Mathematics\\
Harvard University\\
Cambridge, MA 02138}
\email{nate@math.harvard.edu}

\author[Freer]{Cameron Freer}
\address{
	Remine\\
	Fairfax, VA 22031}
\email{cameron@remine.com}

\author[Kruckman]{\\Alex Kruckman}
\address{Department of Mathematics\\
Indiana University\\
Bloomington, IN 47405}
\email{akruckma@indiana.edu}

\author[Patel]{Rehana Patel}
\address{Wheaton College\\ 
Norton, MA\\
02766}
\email{patel\_rehana@wheatoncollege.edu}


\begin{abstract}
We consider ergodic $\sym$-invariant probability measures on the space of $L$-structures with domain $\Nats$ (for $L$ a countable relational language), and call such a measure a \emph{properly ergodic structure} when no isomorphism class of structures is assigned measure $1$.  We characterize those theories in countable fragments of $\Lwow$ for which there is a properly ergodic structure concentrated on the models of the theory.  We show that for a countable fragment $F$ of $\Lwow$ the almost-sure $F$-theory of a properly ergodic structure has continuum-many models (an analogue of Vaught's Conjecture in this context), but its full almost-sure $\Lwow$-theory has no models.  We also show that, for an $F$-theory $T$, if there is some properly ergodic structure that concentrates on the class of models of $T$, then there are continuum-many such properly ergodic structures.
\end{abstract}

\maketitle


\setcounter{page}{1}
\thispagestyle{empty}

\begin{small}
\renewcommand\contentsname{\!\!\!\!}
\setcounter{tocdepth}{3}
\tableofcontents
\end{small}


\section{Introduction}

Symmetric random constructions of mathematical structures have been extensively studied in probability theory, combinatorics, and logic.  One of the best-known examples is the countably infinite \ER\ random graph, whose edges are determined by an independent coin flip for each pair of vertices.  With probability 1, this process produces a particular countable graph up to isomorphism, known as the Rado graph (or random graph). 
The paper \cite{AFP} provided a characterization of those countable structures that can be produced via a symmetric random construction. 

In the present paper we study the \emph{properly ergodic structures}, namely those symmetric random constructions that do not give rise to a single mathematical structure, but rather spread their probability mass across many isomorphism classes of structures.  To do so, we make use of tools from infinitary model theory and from probability theory, especially the Aldous--Hoover--Kallenberg representation of exchangeable structures.

\subsection{Ergodic structures}
Let $L$ be a countable relational language, and
write $\Str_L$ for the measurable space of $L$-structures with domain $\n$.
We say that a probability measure on $\Str_L$ is \emph{invariant} when it is
invariant under the natural action (called the \emph{logic action}) of the permutation group $\sym$ on $\Str_L$.  An invariant probability measure
on $\Str_L$ can be thought of as a distribution on countable structures that
does not depend on the labeling of the domain.
The orbits of the logic action are the isomorphism classes of $L$-structures in $\Str_L$.

An invariant probability measure $\mu$ on $\Str_L$ is \emph{ergodic} when the null and co-null sets are the only Borel sets that are almost surely invariant. In other words, $\mu$ is ergodic if, whenever $\mu(X\triangle \sigma[X]) = 0$ for all $\sigma\in \sym$, we have $\mu(X) = 0$ or $\mu(X) = 1$. 
The ergodic invariant probability measures are extreme points in the space of invariant probability measures on $\Str_L$, and any invariant probability measure can be decomposed as a mixture of ergodic ones. (For details, see, e.g., \cite[Lemma~A1.2 and Theorem~A1.3]{MR2161313}.) 
Hence when considering invariant probability measures on $\Str_L$,
it often suffices to restrict attention to the ergodic ones.

In fact, the ergodic invariant probability measures can be thought of as random symmetric analogues of model-theoretic structures, and so we call them
\emph{ergodic structures}.
An ergodic structure $\mu$ determines the ``almost-sure truth value" of every sentence of the infinitary logic $\Lwow$, as the set of models for a sentence of $\Lwow$ is an invariant Borel set in $\Str_L$, and hence is assigned measure $0$ or $1$ by $\mu$. 
Therefore every ergodic structure has a complete almost-sure theory,
in $\Lwow$ or in any fragment of $\Lwow$.

The ergodic structures have several additional nice properties.
The Aldous--Hoover--Kallenberg theorem 
implies that every invariant probability measure on $\Str_L$ can be represented as a random process that depends on independent sources of randomness at every finite subset of $\n$ (see \S\ref{sec:aldous-hoover} for more details). The ergodic structures are those invariant measures with dissociated representations, i.e., in which the random process does not depend on ``global'' randomness (formally, randomness indexed by the empty set in the representation). Equivalently, ergodic structures are those in which the behavior on disjoint finite subsets of $\n$ is independent.

Further, the ergodic structures are exactly those invariant measures which arise as a limit, in the weak topology, of measures obtained by uniformly sampling from a finite structure.  If we restrict to a language with a single binary relation, a rich source of ergodic structures comes from those measures obtained by sampling a graphon \cite{Lovasz}. Just as graphons arise as limits of sequences of finite graphs which are convergent in the appropriate sense, ergodic structures can be viewed as limits of convergent sequences of finite $L$-structures. For details, see~\cite[\S 1.2]{kruckman-thesis}.

\subsection{Properly ergodic structures}
\label{intro-subsec-propergodic}
An ergodic structure is called \emph{properly ergodic} when it does not assign measure $1$ to any $\sym$-orbit, i.e., isomorphism class of $L$-structures. In fact, a properly ergodic structure must assign measure $0$ to every $\sym$-orbit.

The paper \cite{AFPcompleteclassification} characterized those $F$-theories (where $F$ is a countable fragment of $\Lwow$) that are the complete almost-sure $F$-theory of an ergodic structure. This characterization was in terms of trivial definable closure (see \S\ref{sec:dcl}), generalizing the result in \cite{AFP} for the case of a single $\sym$-orbit (i.e., the non-properly ergodic case).
In the present paper, we are interested in understanding which $F$-theories are the complete almost-sure $F$-theory of a properly ergodic structure.

The most well-known examples of ergodic structures (such as the \ER\ random graph described above) concentrate on a single isomorphism class, and
indeed, it is not immediately obvious how to construct any properly ergodic structures.
At the American Institute of Mathematics workshop on \emph{Graph and Hypergraph Limits} in 2011, Omer Angel 
asked whether 
the distribution on countable graphs 
induced by a graphon can 
have more than one isomorphism class in its support.
By \cite{MR2921000}, the distribution on countable graphs
induced by a graphon is ergodic, and so
this question is asking 
whether there are any properly ergodic structures that concentrate on the theory of graphs.
During the workshop G\'abor Kun provided an example of such a properly ergodic structure;
see \cite[\S2.3]{aim-workshop-notes} for details. 
In fact, a properly ergodic graph was discovered somewhat earlier by Bonato and Janssen in \cite{GeometricGraphs}; see 
Example~\ref{ex:geometric} for details.

The case of properly ergodic structures has been further considered in \cite{AFNP}, where a class of examples was constructed that concentrate on the sets of models of certain ``approximately $\aleph_0$-categorical'' first-order theories with trivial definable closure. 
One such class of examples is described in Example~\ref{ex:kaleidoscopes}.

In the present paper we characterize
those $F$-theories that are the complete almost-sure $F$-theory of a properly ergodic structure.
Additionally, we show that for any properly ergodic structure $\mu$, the complete almost-sure $\Lwow$-theory of $\mu$ has no models (of any cardinality), but that for any countable fragment $F$, the complete almost-sure $F$-theory of $\mu$ has continuum-many models up to isomorphism. This can be viewed as an analogue of Vaught's Conjecture in the setting of ergodic structures.

In~\cite{AFKP} it was shown that for every countable structure $M$, the number of ergodic structures concentrating on its isomorphism class is zero, one, or continuum.  Moreover, the case of one only occurs when $M$ is highly homogeneous, i.e., interdefinable with one of the five reducts of the rational linear order. We extend this result to the properly ergodic case, showing that if there exists a properly ergodic structure concentrating on the class of models of an $F$-theory $T$, then there are continuum-many properly ergodic structures concentrating on the class of models of $T$.

\subsection{Outline of paper}

Section~\ref{sec:background} contains some basic definitions and results, including the notion of trivial definable closure, a particular form of $\Pi_2$ sentence that we call $\Pi_2^-$, and the Aldous--Hoover--Kallenberg representation. 

In Section~\ref{sec:examples}, we provide a number of examples of properly ergodic structures, which illustrate some of their key features. 

In Section~\ref{sec:analysis} we undertake a Morley--Scott analysis of an ergodic structure $\mu$, based on Morley's proof \cite{Morley} that the number of isomorphism classes of countable models of a sentence of $\Lwow$ is countable, $\aleph_1$, or $2^{\aleph_0}$. This gives us a notion of Scott rank for ergodic structures and, in the properly ergodic case, allows us to find a countable fragment $F$ of $\Lwow$ in which there is a formula $\chi(\xx)$ which is satisfied with positive probability (under instantiations of its parameters independently sampled from $\mu$), but which picks out continuum-many $F$-types, each of which has probability $0$ of being realized. The analogue of Vaught's Conjecture mentioned above is a corollary of this analysis.

In Section~\ref{sec:rootedness}, we introduce the notion of a \emph{rooted} model of a theory. A structure $M$ is rooted if every collection of non-isolated types (e.g., the continuum-many types of measure $0$ coming from the Morley--Scott analysis) has ``few'' realizations in $M$, in a sense that we will make precise. We use the Aldous--Hoover--Kallenberg theorem to show that a structure sampled from a properly ergodic measure is almost surely rooted.

In Section~\ref{sec:construction}, given a theory $T$ having trivial definable closure, we use a single rooted model of $T$ to guide the construction, via an inverse limit, of a rooted Borel model $\bbM\models T$ equipped with an atomless probability measure $\nu$.  By sampling from $(\bbM,\nu)$, we obtain a properly ergodic structure $\mu$ that concentrates on the class of models of $T$.  The inverse limit construction is a refinement of the methods from~\cite{AFP}, \cite{AFNP}, and~\cite{AFPcompleteclassification}, which in turn generalized a construction of Petrov and Vershik~\cite{PV}. Further, we use a technique from~\cite{AFKP} to rescale $\nu$, obtaining continuum-many properly ergodic structures concentrating on the class of models of $T$. 

Putting together the results of Sections~\ref{sec:analysis}--\ref{sec:construction}, we obtain the characterization of the complete almost-sure $F$-theories of properly ergodic structures.

\section{Preliminaries}\label{sec:background}

\subsection{The space \texorpdfstring{$\Str_L$}{Str\_L}, infinitary logic, and ergodic structures}\label{sec:definitions}

Throughout this paper, let $L$ be a countable relational language. We study invariant measures on the space $\Str_L$ of $L$-structures with domain $\n$. One could formulate this work in terms of arbitrary countable languages (which allow constant and function symbols), but it turns out that one does not lose much by working in the relational case --- there are no ergodic structures in languages having constant symbols, and, in an ergodic structure, the interpretation of a function symbol must take some value among its inputs almost surely (see \cite[\S\S3--4]{AFP}). For more details on how to translate results about invariant measures for countable relational languages to the case of arbitrary countable languages, see \cite{AFPcompleteclassification}.

\begin{definition}\label{def:Str_L}
$\Str_L$ is the space of $L$-structures with domain $\n$. The topology is generated by the sets of the form $\extent{R(\aa)} = \{M\in\Str_L\mid M\models R(\aa)\}$ and $\extent{\lnot R(\aa)} = \{M\in\Str_L\mid M\models \lnot R(\aa)\}$, where $R$ ranges over the relation symbols in $L$ and $\aa$ ranges over the $\ar(R)$-tuples from $\n$.
\end{definition}

A structure $M\in\Str_L$ is uniquely determined by whether or not, for each
relation symbol $R$ in $L$ of arity $\ar(R)$ and each $\ar(R)$-tuple $\aa$ from $\n$,
\[
	M\models R(\aa)
\]
holds. It follows that $\Str_L$ is homeomorphic to the Cantor space
\[
	\prod_{R\in L} 2^{(\n^{\ar(R)})}.
\]

Recall that $\Lwow$ is the infinitary extension of first-order logic obtained by allowing, as new formula-building operations, the conjunction or disjunction of any countable ($<\omega_1$) family of formulas with a common finite ($<\omega$) set of free variables. We ensure that all our variables come from a fixed countable supply. For a reference on $\Lwow$, see~\cite{MR2062240}.

In contrast to the infinitary logic $\Lwow$, we will also be interested in the quantifier-free fragment of first-order logic, in which the only formula-building operations are negation, finite conjunction, and finite disjunction. Throughout this paper, when we speak of quantifier-free formulas and types, we mean quantifier-free first-order formulas and types.

Given a formula $\varphi(\xx) \in \Lwow$ and a tuple $\aa$ from $\mathbb{N}$ of the same length as $\xx$, we let 
\[
	\extent{\varphi(\aa)} = \{M\in \Str_L\mid M\models \varphi(\aa)\}.
\]
Every $\extent{\varphi(\aa)}$ is a Borel set in $\Str_L$. Indeed, the formula-building operations of negation and countable conjunction and disjunction correspond to the set-building operations of complementation and countable intersection and union, and quantifiers over the countable domain also correspond to certain countable intersections and unions: 
\begin{align*}
\extent{\forall x\,\varphi(\aa,x)} &= \bigcap_{b\in \n}\extent{\varphi(\aa,b)}\\
\extent{\exists x\,\varphi(\aa,x)} &= \bigcup_{b\in \n}\extent{\varphi(\aa,b)}.
\end{align*}

Restricting ourselves to finite Boolean operations, $\extent{\varphi(\aa)}$ is a clopen set when $\varphi(\xx)$ is a quantifier-free formula. In fact, by compactness, every clopen set in $\Str_L$ has the form $\extent{\varphi(\aa)}$ for some quantifier-free formula $\varphi$.

Let $\sym$ denote the permutation group of $\n$.
\begin{definition} \label{def:logicaction}
The \defn{logic action} is the natural action of $\sym$ on $\Str_L$, given by permuting the underlying set. Namely, for $\sigma\in \sym$ and $M\in\Str_L$, we have
\[
	\sigma(M)\models R(a_1,\dots,a_n)\quad\text{ if and only if } \quad M\models R\bigl(\sigma^{-1}(a_1), \dots, \sigma^{-1}(a_n)\bigr)
\]
for all $R \in L$.
\end{definition}

Note that $\sigma(M) = N$ if and only if $\sigma\colon M\to N$ is an isomorphism of $L$-structures, so the orbit of a point $M\in \Str_L$ under the logic action is the set of all structures in $\Str_L$ which are isomorphic to $M$. We recall Scott's theorem, which says that this set is definable by a sentence of $\Lwow$. 

\begin{theorem}[Scott, {\cite[Theorem~2.4.15]{Marker}}]\label{thm:scott}
For any countable structure $M$, there is a sentence $\varphi_M$ of $\Lwow$, the \defn{Scott sentence of $M$}, such that for all countable structures $N$, we have $N\models \varphi_M$ if and only if $N\cong M$. \end{theorem}

We are now able to define the class of invariant measures on $\Str_L$, and specifically, the ergodic and properly ergodic ones.

\begin{definition}\label{def:measures}
Let $\mu$ be a Borel probability measure on $\Str_L$. 
We say that $\mu$ is \defn{invariant} (under the logic action) if, for every Borel set $X$ and every $\sigma\in\sym$, we have $\mu(\sigma[X]) = \mu(X)$. 

Now suppose that $\mu$ is invariant.
A Borel set $X$ is \defn{almost surely invariant} if $\mu(X \triangle \sigma[X]) = 0$ for all $\sigma\in\sym$. 
We say that $\mu$ is \defn{ergodic} if, for every almost surely invariant Borel set $X$, either $\mu(X) = 0$ or $\mu(X) = 1$. 
Following terminology from \cite[\S I.2]{MR1781937} and elsewhere, we say that $\mu$ is \defn{properly ergodic} if $\mu(X)= 0$ for every orbit $X$ of the logic action.
\end{definition}

\begin{definition} \label{def:ergodicstructure}
An \defn{ergodic structure} is an ergodic invariant probability measure on $\Str_L$. 
\end{definition}

This definition takes on a more concrete character if we restrict our attention to the measures assigned to instances of quantifier-free formulas. The following proposition is an application of the Hahn--Kolmogorov measure extension theorem \cite[Theorem 1.7.8, Exercise 1.7.7]{Tao}.

\begin{proposition}\label{prop:stonemeasure}
Let $\mathcal{B}^*$ be the Boolean algebra of clopen sets in $\Str_L$ (so $\mathcal{B}^*$ consists of those sets of the form $\extent{\varphi(\aa)}$, where $\varphi$ is a quantifier-free formula and $\aa$ is a tuple from $\n$). Any finitely additive measure $\mu^*$ on $\mathcal{B}^*$ extends to a unique Borel probability measure $\mu$ on $\Str_L$. Moreover, $\mu$ is invariant if and only if $\mu^*$ is; that is, if and only if $\mu^*(\extent{\varphi(\aa)}) = \mu^*(\extent{\varphi(\sigma(\aa))})$ for any $\sigma\in \sym$.
\end{proposition}

\begin{remark}\label{rem:ergodic}
Additionally, it follows from Theorem~\ref{thm:ergodic} below that an invariant measure $\mu$ on $\Str_L$ is ergodic if and only if the quantifier-free types of disjoint tuples from $\n$ are independent. That is, whenever $\varphi(\xx)$ and $\psi(\yy)$ are quantifier-free formulas and $\aa$ and $\bb$ are disjoint tuples from $\n$ whose lengths are those of $\xx$ and $\yy$ respectively, we have $\mu(\extent{\varphi(\aa)\land \psi(\bb)}) = \mu(\extent{\varphi(\aa)})\,\mu(\extent{\psi(\bb)})$. 
\end{remark}

\textbf{For the remainder of this subsection, let $\mu$ be an ergodic structure.}

\begin{remark}\label{measure-of-formula}
	If $\varphi(\xx)$ is a formula of $\Lwow$ and $\aa$ is a tuple of distinct elements of $\n$ (of the same length as $\xx$), then, since $\mu$ is invariant under the logic action, the value $\mu(\extent{\varphi(\aa)})$ is independent of the choice of $\aa$. For convenience, we denote this quantity by $\mu(\varphi(\xx))$, and refer to it as the \emph{measure of the formula $\varphi$}. Note that under this convention, if $\varphi(\xx)$ implies $x_i = x_j$ for some $i\neq j$, then $\mu(\varphi(\xx)) = 0$.
\end{remark}

\begin{definition}\label{def:modelsas}
If $\varphi$ is a sentence of $\Lwow$, we say $\mu$ \defn{almost surely satisfies} $\varphi$, or $\mu$ \defn{concentrates on} $\varphi$, if $\mu(\varphi) = 1$. We write $\mu\modelsas \varphi$, and we set
\[
	\Th(\mu) = \{\varphi\in\Lwow\mid \mu\modelsas\varphi\}.
\]
Similarly, if $\Sigma$ is a set of sentences of $\Lwow$, we write $\mu\modelsas \Sigma$ if $\mu\modelsas \varphi$ for all $\varphi\in \Sigma$, and say that $\mu$ is an \defn{ergodic model} of $\Sigma$.
\end{definition}

The following result is a connection between infinitary logic and ergodic invariant measures; see also \cite{AFPcompleteclassification}.

\begin{proposition}\label{prop:complete}
$\Th(\mu)$ is a complete and countably consistent theory of $\Lwow$. That is, for every sentence $\varphi$ of $\Lwow$, $\varphi\in \Th(\mu)$ or $\lnot\varphi\in \Th(\mu)$, and every countable subset $\Sigma\subseteq \Th(\mu)$ has a model.
\end{proposition}
\begin{proof}
For any sentence $\varphi$, the set $\extent{\varphi}$ is an invariant Borel set. In particular, it is almost surely invariant, so by ergodicity, $\mu(\varphi) = 0$ or $1$, and hence $\mu\modelsas \varphi$ or $\mu\modelsas \lnot\varphi$. Now let $\Sigma$ be a countable subset of $\Th(\mu)$. Since a countable intersection of measure $1$ sets has measure $1$, $\mu(\bigdoublewedge_{\varphi\in\Sigma}\varphi) = 1$. In particular, $\extent{\bigdoublewedge_{\varphi\in\Sigma}\varphi}$ is non-empty.
\end{proof}

A special case of Definition~\ref{def:modelsas} is when the sentence $\varphi$ is a Scott sentence.

\begin{definition}\label{def:asisomorphic}
If $M$ is a countable structure, we say that $\mu$ is \defn{almost surely isomorphic to} $M$, or $\mu$ \defn{concentrates on} $M$, if $\mu\modelsas \varphi_M$, where $\varphi_M$ is the Scott sentence of $M$; equivalently, $\mu$ assigns measure $1$ to the orbit of $M$. \end{definition}

\begin{remark}\label{rem:negscott}
If $\mu$ is properly ergodic, then $\Th(\mu)$ contains $\lnot\varphi_M$ for every countable structure $M$, and thus $\Th(\mu)$ has no countable models. A priori, $\Th(\mu)$ may have uncountable models (L\"owenheim--Skolem does not apply to complete theories of $\Lwow$), but we will see later (Corollary~\ref{cor:nomodels}) that this is not the case: $\Th(\mu)$ has no models of any cardinality. Nevertheless, as noted in Proposition~\ref{prop:complete}, every countable subset of $\Th(\mu)$ has countable models. This suggests that we should restrict our attention to countable fragments of $\Lwow$. 
\end{remark}

\begin{definition}\label{def:fragment}
 A \defn{fragment} of $\Lwow$ is a set of formulas which contains all atomic formulas and is closed under subformulas, finite Boolean combinations, quantification, and substitution of free variables (from the countable supply). If $F$ is a fragment of $\Lwow$, we set 
 \[
	 \Th_F(\mu) = \{\varphi\in F\mid \mu\modelsas \varphi\}.
 \]
\end{definition}

A countable set of formulas $\Phi$ generates a countable fragment $\langle \Phi\rangle$, the least fragment containing this set. The minimal fragment $\FO \defas \langle \emptyset \rangle$ is first-order logic.

\begin{definition}\label{def:Fstuff}
Let $F$ be a countable fragment of $\Lwow$. 
\begin{itemize}
\item A set of sentences $T$ is a (complete satisfiable) \defn{$F$-theory} if $T$ has a model and, for every sentence $\varphi\in F$, either $\varphi\in T$ or $\lnot \varphi\in T$. Equivalently, 
	there is a structure $M$ for which
	$T = \{\psi\in F\mid M\models \psi\}$.
\item A set of formulas $p(\xx)$ is an \defn{$F$-type} if there is a structure $M$ and a tuple $\aa$ from $M$ such that $p(\xx) = \{\psi(\xx)\in F\mid M\models \psi(\aa)\}$. We say that $\aa$ \defn{realizes} $p$ in $M$. 
\item An $F$-type $p$ is \defn{consistent} with an $F$-theory $T$ if it is realized in some model of $T$, and we write $S^n_F(T)$ for the set of $F$-types in $n$ variables which are consistent with $T$.
\end{itemize}
\end{definition}

\begin{remark}\label{rem:countablemodels}
The L\"{o}wenheim--Skolem theorem holds for countable fragments of $\Lwow$ (see~\cite[Theorem 1.5.4]{MR2062240}). Thus, if $F$ is countable, every $F$-theory has a countable model and every $F$-type which is consistent with $T$ is realized in a countable model of $T$.
\end{remark}

\begin{remark}\label{measure-of-type}
	If $F$ is countable and $p$ is an $F$-type, then we denote by $\theta_p(\xx)$ the conjunction of all the formulas in $p$, i.e., $\doublewedge_{\varphi\in p}\varphi(\xx)$. This is a formula of $\Lwow$ (although not a formula of $F$ in general), so it is assigned a measure by our ergodic structure $\mu$, as described in Remark~\ref{measure-of-formula}. We will write $\mu(p)$ as shorthand for $\mu(\theta_p(\xx))$, and refer to this as the \emph{measure of the type $p$}. This is the probability, according to $\mu$, that any given tuple of distinct elements of $\n$ (of the appropriate arity) satisfies $p$.
\end{remark}

\subsection{Trivial definable closure}\label{sec:dcl}

The paper \cite{AFPcompleteclassification} shows that trivial definable closure is 
a necessary and sufficient condition for a theory (in a countable fragment) to have an an ergodic structure which satisfies it.
Here we state several definitions and basic facts, and we provide a proof of one direction of this characterization.

\begin{definition}\label{def:dcl}
	Let $F$ be a fragment of $\Lwow$. An $F$-theory $T$ has \defn{trivial definable closure} (abbreviated \defn{trivial dcl})
	if there is no formula $\varphi(\xx,y)$ in $F$ such that 
\[
\textstyle
T \models
\exists \xx\, \exists ! y\, \bigl((\bigwedge_{i=1}^n y\neq x_i) \land \varphi(\xx,y) \bigr).
\]
Here $\exists !\,y$ is the standard abbreviation for ``there exists a unique $y$''.
\end{definition}

\begin{remark}\label{rem:dcl}
If $T$ is the complete $F$-theory of a structure $M$, then $T$ has trivial dcl if and only if $M$ has trivial dcl for the fragment $F$ in the usual sense: $\dcl_F(A) = A$ for all $A\subseteq M$, where $\dcl_F(A)$ is the set of all $b\in M$ such that $b$ is the unique element of $M$ satisfying some formula in $F$ with parameters from $A$.

If $\varphi(\xx,y)$ witnesses that $T$ has nontrivial dcl, then taking $\varphi^*$ to be the formula $\varphi(\xx,y) \land \exists^{\leq 1}y\, \varphi(\xx,y)$ we have the stronger condition that $T$ proves that $\varphi^*$ is a definable function on some non-empty domain. That is, 
\[
	\textstyle
	T \models
	\left(\exists\xx\,\exists y\, (\bigwedge_{i = 1}^n y\neq x_i) \land \varphi^*(\xx,y)\right)\land \left(\forall \xx\, \exists^{\leq 1} y\, \varphi^*(\xx,y)\right).
\]
Here $\exists^{\leq 1}y$ is the standard abbreviation for ``there is at most one $y$''.
\end{remark}

The following argument first appeared (in a slightly different setting) in \cite[Theorem~4.1]{AFP}; as stated, this result is from \cite{AFPcompleteclassification}. We include it here for completeness.

The key observation is the standard fact that
if a measure is invariant under the action of some group $G$, then no positive-measure set can have infinitely many almost surely disjoint images under the action of $G$.

\begin{theorem}\label{thm:trivialdcl} 
Let $\mu$ be an ergodic structure and $F$ a fragment of $\Lwow$. Then $\Th_F(\mu)$ has trivial dcl. 
\end{theorem}
\begin{proof}
Suppose there is a formula $\varphi(\xx,y)$ in $F$ such that 
\[
	\textstyle\mu\left(\exists\xx\,\exists! y\, \bigl(\left(\bigwedge_{i=1}^n y\neq x_i\right) \land \varphi(\xx,y)\bigr)\right) = 1.
\]
Let $\psi(\xx,y)$ be the formula $\left(\bigwedge_{i=1}^n y\neq x_i\right) \land \varphi\left(\xx,y\right)$. 

By countable additivity of $\mu$, there is a tuple $\aa$ from $\n$ such that 
\[
	\mu(\extent{\exists! y\, \psi(\aa,y)}) > 0.
\]

Let $\theta(\aa)$ be the formula $\forall z_1\,\forall z_2\,(\psi(\aa,z_1)\land \psi(\aa,z_2)\rightarrow(z_1 = z_2))$, so that $\exists! y\, \psi(\aa,y)$ is equivalent to 
\[
	\exists y\, \psi(\aa,y) \land \theta(\aa).
\]
 
Since this formula has positive measure, countable additivity again implies that there is some $b\in\n\setminus \aa$ such that 
\[
	\beta \defas \mu(\extent{\psi(\aa,b) \land \theta(\aa)}) > 0.
\]

By invariance, for any $c\in\n\setminus\aa$, we also have 
\[
	\mu(\extent{\psi(\aa,c) \land \theta(\aa)}) = \beta.
\]

But $\theta$ ensures that $\psi(\aa,b) \land \theta(\aa)$ and $\psi(\aa,c) \land \theta(\aa)$ are inconsistent when $b\neq c$, so, computing the measure of the disjoint union,
\[
	\textstyle \mu\left(\bigcup_{b\in\n\setminus \aa} \extent{\psi(\aa,b) \land \theta(\aa)}\right) = \textstyle \sum_{b\in\n\setminus \aa} \beta =\infty,
\]
which is impossible.
\end{proof}

In the language of \S\ref{sec:definitions}, the main result of \cite{AFP} was a characterization of those countable structures $M$ such that there exists an ergodic structure $\mu$ which is almost surely isomorphic to $M$. That characterization was given in terms of trivial ``group-theoretic'' dcl (where the group is $\Aut(M)$).

\begin{definition}\label{def:groupdcl}
A countable structure $M$ has \defn{trivial group-theoretic dcl} if for any finite subset $A\subseteq M$ and element $b\in M\setminus A$, there is an automorphism $\sigma\in \Aut(M)$ such that $\sigma(a) = a$ for all $a\in A$, but $\sigma(b)\neq b$.
\end{definition}

\begin{theorem}[{\cite[Theorem~1.1]{AFP}}]\label{thm:AFP1}
Let $M$ be a countable structure. There exists an ergodic structure concentrating on $M$ if and only if $M$ has trivial group-theoretic dcl.
\end{theorem}

\begin{remark}\label{rem:AFPergodic}
The method in \cite{AFP} of obtaining a measure via i.i.d.\ sampling from a Borel structure, which we use again in Section~\ref{sec:construction}, always produces an ergodic measure. This was mentioned in passing in \cite{AFP}, though not stated as part of the main theorem; for a proof, see \cite[Proposition 2.24]{AFKP}. See also Theorem~\ref{thm:ergodic} and Lemma~\ref{lem:borelsampling} below.
\end{remark}

It is a consequence of Scott's Theorem (Theorem~\ref{thm:scott}) that the notion of trivial group-theoretic dcl for a countable structure $M$ is equivalent to the usual (syntactic) trivial dcl for $\Th_{F_M}(M)$ in an appropriate countable fragment $F_M$ of $\Lwow$. That is, given a finite subset $A$ of $M$, an element $b\in M$ is fixed by all automorphisms fixing $A$ pointwise if and only if there is a formula from $F_M$ with parameters from $A$ which uniquely defines $b$ in $M$.

Unlike the group-theoretic notion of trivial dcl, which is defined for a given structure, the syntactic notion of trivial dcl (Definition~\ref{def:dcl}) is defined for theories in arbitrary countable fragments, and so is the relevant notion for this paper.

\subsection{\texorpdfstring{$\Pi_2^-$}{Pi\_2} sentences}\label{sec:pithy}

It is a well-known fact, originally due to Chang \cite[pp.~48--49]{ChangMR0234827}, that if $T$ is a theory in a countable fragment $F$ of $\Lwow$, then the models of $T$ are exactly the reducts to $L$ of the models of a countable first-order theory $T'$ in a larger countable language $L'\supseteq L$ that omit a countable set of types $Q$.

The idea is to Morleyize: we introduce a new relation symbol $R_\varphi$ for every formula $\varphi(\xx)$ in $F$ and encode the intended interpretations of the $R_\varphi$ in the theory $T'$. The role of the countable set of types $Q$ is to achieve this for infinitary conjunctions and disjunctions, which cannot be accounted for in first-order logic. 

There are two features of this construction that will be useful for us. First, it reduces $F$-types to quantifier-free types. Second, $T'$ can be axiomatized by 
$\Pi_1$ sentences together with
pithy $\Pi_2$ sentences (also called ``one point extension axioms'').

\begin{definition}\label{def:pithy}
A first-order sentence is \defn{pithy $\Pi_2$} if it has the form $\forall \xx\, \exists y\, \varphi(\xx,y)$, where $\varphi(\xx,y)$ is quantifier-free, $\xx$ is a tuple of variables (possibly empty), and $y$ is a single variable. We call a sentence $\Pi_2^-$ if it is either pithy $\Pi_2$ or is $\Pi_1$. A \defn{$\Pi_2^-$ theory} is a set of $\Pi_2^-$ sentences.
\end{definition}

Note that, in the context of this paper, all $\Pi_2^-$ theories are first-order.

\begin{theorem}\label{thm:pithy}
Let $F$ be a countable fragment of $\Lwow$ and $T$ an $F$-theory. Then there is a language $L'\supseteq L$, an $L'$-theory $T'$ that is $\Pi_2^-$, and a countable set of partial quantifier-free $L'$-types $Q$ such that the following hold.
\begin{enumerate}[(a)]
\item There is a bijection between formulas $\varphi(\xx)$ in $F$ and atomic $L'$-formulas $R_\varphi(\xx)$ which are not in $L$, such that if $M\models T'$ omits all the types in $Q$, then $M\models \forall \xx\, \varphi(\xx)\leftrightarrow R_\varphi(\xx)$. 
\item The reduct to $L$ is a bijection between the class of models of $T'$ omitting all the types in $Q$ and the class of models of $T$.
\end{enumerate}
\end{theorem}
\begin{proof}
Let $L' = L\cup \{R_{\varphi}\mid \varphi(\xx)\in F\}$, where the arity of the relation symbol $R_\varphi$ is the length of the tuple $\xx$. By convention, we allow $0$-ary relation symbols (i.e., propositional symbols). Thus, we include a $0$-ary relation $R_\psi$ for every sentence $\psi\in F$. 

Let $T_\tdef$ be the theory consisting of the following axioms, for each formula $\varphi(\xx)\in F$:
\begin{enumerate}
	\item $\forall \xx\, \bigl(R_\varphi(\xx) \leftrightarrow \varphi(\xx)\bigr)$, if $\varphi(\xx)$ is atomic.
	\item $\forall \xx\, \bigl(R_{\varphi}(\xx) \leftrightarrow \lnot R_{\psi}(\xx)\bigr)$, if $\varphi$ is of the form $\lnot \psi(\xx)$.
	\item $\forall \xx\, \bigl(R_{\varphi}(\xx) \leftrightarrow R_{\psi}(\xx) \land R_{\theta}(\xx) \bigr)$, if $\varphi$ is of the form $\psi(\xx)\land \theta(\xx)$.
	\item $\forall \xx\, \bigl(R_{\varphi}(\xx) \leftrightarrow R_{\psi}(\xx) \lor R_{\theta}(\xx)\bigr)$, if $\varphi$ is of the form $\psi(\xx)\lor \theta(\xx)$.
	\item $\forall \xx\, \bigl(R_{\varphi}(\xx) \rightarrow R_{\psi_i}(\xx)\bigr)$ for all $i\in I$, if $\varphi$ is of the form $\doublewedge_{i\in I} \psi_i(\xx)$.
	\item $\forall \xx\, \bigl(R_{\psi_i}(\xx) \rightarrow R_{\varphi}(\xx)\bigr)$ for all $i \in I$, if $\varphi$ is of the form $\doublevee_{i\in I}\psi_i(\xx)$.
	\item $\forall \xx\, \bigl(R_{\varphi}(\xx) \leftrightarrow \forall y\, R_{\psi}(\xx,y)\bigr)$, if $\varphi$ is of the form $\forall y\, \psi(\xx,y)$.
	\item $\forall \xx\, \bigl(R_{\varphi}(\xx) \leftrightarrow \exists y\, R_{\psi}(\xx,y)\bigr)$, if $\varphi$ is of the form $\exists y\, \psi(\xx,y)$.
\end{enumerate}
Note that all the axioms of $T_{\tdef}$ are first-order and universal except for those of type (7) and (8), which are $\Pi_2^-$ when put in prenex normal form.

The axioms of type (5) and (6) cannot be made into bi-implications, since arbitrary countable infinite conjunctions and disjunctions are not expressible in first-order logic. To ensure that the corresponding $R_{\varphi}$ have their intended interpretation, we let $Q$ consist of the partial quantifier-free types:
\begin{itemize}
	\item[\emph{(i)}] $q_\varphi(\xx) = \{R_{\psi_i}(\xx)\mid i\in I\} \cup \{\lnot R_{\varphi}(\xx)\}$, for all $\varphi(\xx)$ of the form $\doublewedge_{i\in I}\psi_i(\xx)$
	\item[\emph{(ii)}] $q_\varphi(\xx) = \{\lnot R_{\psi_i}(\xx)\mid i\in I\} \cup \{R_{\varphi}(\xx)\}$, for all $\varphi(\xx)$ of the form $\doublevee_{i\in I}\psi_i(\xx)$.
\end{itemize}

It is now straightforward to show by induction on the complexity of formulas that if a model $M\models T_{\tdef}$ omits every type in $Q$, then for all $\varphi(\xx)$ in $F$ and all $\aa$ from $M$, we have $M\models \varphi(\aa)$ if and only if $M\models R_{\varphi}(\aa)$. This establishes \emph{(a)}. It also implies that every $L$-structure $N$ admits a unique expansion to an $L'$-structure $N'$ which is a model of $T_\tdef$ and omits every type in $Q$.
As a consequence, if we set $T' = T_\tdef\cup \{R_\psi\mid \psi\in T\}$, then the following hold.
\begin{itemize}
\item If $M$ is a model of $T'$ which omits every type in $Q$, then the reduct $M\restriction L$ is a model of $T$. 
\item If $N\models T$, then the canonical expansion $N'$ of $N$ is a model of $T'$.
\item If $M \models T'$ then $(M\restriction L)' = M$.
\item If $N \models T$ then $N'\restriction L = N$.
\end{itemize}
This establishes \emph{(b)}.
\end{proof}

Recall that an ergodic $L$-structure is an ergodic invariant measure on $\Str_L$.

\begin{corollary}\label{cor:measurebijection}
There is a bijection between the invariant measures on $\Str_L$
which almost surely satisfy $T$ and the invariant measures on $\Str_{L'}$
which almost surely satisfy $T'$ and omit all the types in $Q$.
This bijection sends ergodic structures to ergodic structures and properly ergodic structures to properly ergodic structures.
\end{corollary}
\begin{proof}
The reduct $\restriction_L$ is a continuous map $\Str_{L'} \rightarrow \Str_L$, since the preimages of clopen sets in $\Str_L$ are also clopen sets in $\Str_{L'}$. By Theorem~\ref{thm:pithy}, $\restriction_L$ is a bijection between the subspace $X'$ of $\Str_{L'}$ consisting of models of $T'$ which omit all the types in $Q$ and the subspace $X$ of $\Str_L$ consisting of models of $T$. Upon restricting to these subspaces, the inverse of $\restriction_L$ is a Borel map, since the image of a clopen set in $X'$ (described by a quantifier-free formula) is a Borel set in $X$ (described by a formula of $\Lwow$). Hence $\restriction_L$ is a Borel isomorphism between these subspaces, and it induces a bijection between the set of probability measures on $\Str_{L'}$ concentrating on $X'$ and the set of probability measures on $\Str_L$ concentrating on $X$. Moreover, $\restriction_L$ preserves the logic action, so the induced bijection on measures preserves invariance, ergodicity, and proper ergodicity. 
\end{proof}

\subsection{The Aldous--Hoover--Kallenberg theorem and representations}\label{sec:aldous-hoover}

In this section, we state a version of the Aldous--Hoover--Kallenberg theorem. This theorem, which is a generalization of de~Finetti's theorem to exchangeable arrays of random variables, was discovered independently by Aldous \cite{Aldous} and Hoover \cite{Hoover}, and further developed by Kallenberg \cite{MR1182678} and others. For proofs, we direct the reader to Kallenberg's book \cite[Chapter~7]{MR2161313}. See \cite[\S2.5]{Ackerman} for a discussion of how to translate from the purely probabilistic statements in Kallenberg to the setting here, involving spaces of quantifier-free types. 
The survey by Austin \cite{Austin} provides details on its application to random structures.

We denote by $[n]$ the set $\{0,\dots,n-1\}$, by $\n^{[n]}$ the set of $n$-tuples of \emph{distinct} elements of $\n$ (that is, injective functions $[n]\to \n$), and by $\mathcal{P}_\fin(\n)$ the set of all finite subsets of $\n$. Given a tuple $\aa\in \n^{[n]}$, we denote by $\enum{\aa}$ the set in $\mathcal{P}_\fin(\n)$ enumerated by $\aa$.

Let $S^n_\qf(L)$ be the Stone space of quantifier-free $n$-types. Its points are the complete quantifier-free types in the variables $x_0,\dots,x_{n-1}$, and its topology is generated by the clopen sets $\extent{\varphi(\xx)} = \{p(\xx)\in S^n_\qf(L)\mid \varphi\in p\}$ for all quantifier-free formulas $\varphi$. Note that $S^n_\qf(L)$ admits an action of the symmetric group $\sympar{n}$ (the permutation group of $[n]$), by $\sigma(p(x_0,\dots,x_{n-1})) = p(x_{\sigma(0)},\dots,x_{\sigma(n-1)})$ for $\sigma \in \sympar{n}$. We write $S^{[n]}_{\qf}(L)$ for the $\sympar{n}$-invariant subspace of non-redundant quantifier-free types, namely those which contain $x_i\neq x_j$ for all $i\neq j$.

We let $(\xi_A)_{A\in \mathcal{P}_\fin(\n)}$ be a collection of independent random variables, each uniformly distributed on $[0,1]$. We think of $\xi_A$ as a source of randomness sitting on the subset $A$, which we will use to build a random $L$-structure with domain $\n$. If $\aa\in \n^{[n]}$, the injective function $i\colon [n]\to \n$ enumerating $\aa$ associates to each $X\in \mathcal{P}([n])$ a subset $i[X]\subseteq \enum{\aa}$. We denote by $\widehat{\xi}_{\aa}$ the family of random variables $(\xi_{i[X]})_{X\in \mathcal{P}([n])}$.

\begin{definition}
An \defn{AHK system} is a collection of measurable functions 
\[
	(f_n\colon [0,1]^{\mathcal{P}([n])} \to S^{[n]}_{\qf}(L))_{n\in \n}
\]
satisfying the coherence conditions:
\begin{itemize}
\item For all $\sigma\in \sympar{n}$, almost surely $$f_n((\xi_{\sigma[X]})_{X\subseteq [n]}) = \sigma(f_n((\xi_X)_{X\subseteq [n]})).$$
\item For all $0\leq m \leq n$, almost surely $$f_m((\xi_X)_{X\subseteq [m]}) \subseteq f_n((\xi_Y)_{Y\subseteq [n]}).$$
\end{itemize}
\end{definition}

That is, $f_n$ takes as input a collection of values in $[0,1]$, indexed by $\mathcal{P}([n])$, and produces a non-redundant quantifier-free $n$-type. Using our random variables $\xi_A$, we have a natural notion of \defn{sampling} from an AHK system to obtain a non-redundant quantifier-free type $r_{\aa} = f_n(\widehat{\xi}_{\aa})$ for every finite tuple $\aa$ from $\n$. Note that the order in which $\enum{\aa}$ is enumerated by the tuple $\aa$ is significant, since $f_n$ is, in general, not symmetric in its arguments.

The coherence conditions ensure that the quantifier-free types obtained from the function $f_n$ cohere (almost surely), allowing us to define the random structure $\fM$ obtained by \emph{sampling} from the AHK system $(f_n)$. Namely, 
for every tuple $\aa\in\Nats$, 
\[
	\fM \models R(\aa) \qquad \text{if and only if} \qquad R(\xx)\in f_n(\widehat{\xi}_{\aa}),
\]
where $n$ is the length of $\aa$.

One may also directly describe the measure on $\Str_L$ which is the distribution of the random structure $\fM$: an AHK system $(f_n)_{n\in\Nats}$ gives rise to a well-defined finitely-additive probability measure $\mu^*$ on the Boolean algebra $\mathcal{B}^*$ of clopen sets in $\Str_L$, defined by 
\[
	\mu^*(\extent{\varphi(\aa)}) = \lambda^{\mathcal{P}([n])}(f_n^{-1}[\extent{\varphi(\xx)}]),
\]
where $\lambda^{\mathcal{P}([n])}$ is the uniform product measure on $[0,1]^{\mathcal{P}([n])}$. This is the probability that $\varphi(\xx)\in r_{\aa}$, whenever $\aa$ is a tuple of $n$ distinct elements. The coherence conditions imply that this is well-defined: the first ensures that the order in which we list the variables in $\varphi(\xx)$ is irrelevant, and the second ensures that the measure is independent of the variable context $\xx$. 

Since the value of $\mu^*(\extent{\varphi(\aa)})$ does not depend on the choice of tuple $\aa$ of distinct elements, $\mu^*$ is manifestly invariant under the logic action. By Proposition~\ref{prop:stonemeasure}, $\mu^*$ induces a unique invariant Borel probability measure $\mu$ on $\Str_L$. In this case, we say that $(f_n)_{n\in\Nats}$ is an \defn{AHK representation of} $\mu$.

\begin{theorem}[Aldous--Hoover--Kallenberg]
\label{thm:aldous-hoover}
Every invariant probability measure $\mu$ on $\Str_L$ has an AHK representation. 
\end{theorem}

Once a proper translation of notation is applied, Theorem~\ref{thm:aldous-hoover} is equivalent to {\cite[Theorem~7.22]{MR2161313}}, which is usually called the Aldous--Hoover--Kallenberg theorem. For more on such a translation see \cite[\S2.5]{Ackerman}.

The AHK representation produced by Theorem~\ref{thm:aldous-hoover} is not unique, but it is unique up to certain appropriately measure-preserving transformations. See \cite[Theorem 7.28]{MR2161313} for a precise statement.

The key fact to observe about AHK systems is that if $\aa$ and $\bb$ are tuples from $\n$ whose intersection $\enum{\aa} \cap \enum{\bb}$ is enumerated by the tuple $\cc$, then the random quantifier-free types $r_{\aa}$ and $r_{\bb}$ are conditionally independent over $\widehat{\xi}_{\cc}$. If $\aa$ and $\bb$ are disjoint, then $\widehat{\xi}_{\cc} = \xi_\emptyset$.

The Aldous--Hoover--Kallenberg theorem also provides a characterization of the ergodic measures among the invariant measures on $\Str_L$: they are those measures for which the random quantifier-free types $r_{\aa}$ and $r_{\bb}$ are independent when $\aa$ and $\bb$ are disjoint. Formally, for an $n$-tuple $\aa$ from $\n$, let $\Sigma_{\aa}$ be the $\sigma$-algebra on $\Str_L$ generated by the sets $\extent{\varphi(\aa)}$, where $\varphi(\xx)$ ranges over the quantifier-free formulas in the $n$-tuple of variables $\xx$. We say that an invariant probability measure $\mu$ on $\Str_L$ is \defn{dissociated} if whenever $\aa$ and $\bb$ are disjoint tuples from $\n$, the $\sigma$-algebras $\Sigma_{\aa}$ and $\Sigma_{\bb}$ are independent (see Remark~\ref{rem:ergodic} above).

\begin{theorem}[{\cite[Lemma 7.35]{MR2161313}}]\label{thm:ergodic}
Let $\mu$ be an invariant probability measure on $\Str_L$. The following are equivalent:
\begin{enumerate}
\item $\mu$ is ergodic.
\item $\mu$ is dissociated.
\item $\mu$ has an AHK representation in which the functions $f_n$ do not depend on the argument indexed by $\emptyset$.
\end{enumerate}
\end{theorem}

The result \cite[Lemma 7.35]{MR2161313} is stated for finite relational languages, but can be generalized to our setting by a careful modification of the proofs, as described in \cite[Corollary 2.18]{Ackerman}.

\section{Examples}\label{sec:examples}

In this section, we describe some examples of properly ergodic structures and their theories, as well as theories all of whose ergodic models are not properly ergodic. In doing so,
we highlight some of the key notions of the paper, including trivial definable closure and rootedness, and the relevance of infinitary logic. Certain examples are naturally described using infinite languages, but for some we also describe how
they may be framed in terms of finite languages.

When we say that we pick a random element $A\in 2^\n$, we always refer to the uniform (Lebesgue) measure on $2^\n$, the infinite product of the $\Bernoulli(\onehalf)$ measure on $2 = \{0,1\}$. We identify such an $A\in 2^\n$ with both a subset of $\n$ and an infinite binary sequence.

Our first example of a class of properly ergodic structures arose naturally in the study of random graphs.

\begin{example}[Random geometric graphs]\label{ex:geometric}
Consider a metric space $(X,d)$, a probability measure $m$ on $X$, and a real number $p \in (0,1)$.  Bonato and Janssen \cite{GeometricGraphs} define the random geometric graph given by first sampling an $m$-i.i.d.\ sequence of vertices $D \subseteq X$ and then connecting two points $x,y\in D$ such that $d(x,y) < 1$ by an edge or not based on an independent weight-$p$ coin flip. The distribution of this random construction is an ergodic structure.  

Bonato and Janssen showed that when the metric space is $\ell_\infty^n$ for some $n$, the random geometric graph is almost surely isomorphic to a single countable graph, but that on the other hand, the Euclidean plane yields a properly ergodic structure.  In fact, as shown in \cite{GeometricGraphs2}, every normed linear space other than $\ell_\infty^n$ yields a properly ergodic structure.
\end{example}

The next class of examples, which was introduced in \cite[\S5.1]{AFNP}, can be thought of as countably many overlaid instances of the \ER\ random (hyper-)graphs. These are some of the key examples of properly ergodic structures, based on which we also will build several variants.

\begin{example}[Kaleidoscope structures]\label{ex:kaleidoscopes} 
	We begin by describing the case of binary relations.
Let $L = \{R_n\mid n\in\n\}$, where each $R_n$ is a binary relation symbol. The interpretation of each $R_n$ will be irreflexive and symmetric; one may think of each $R_n$ as a different ``color'' of edge.

	Consider the random $L$-structure with domain $\n$ obtained by first picking a random $A_{\{i,j\}}\in 2^\n$ independently for each pair of distinct elements $i,j \in \n$, 
	and then setting $iR_n j$ just when $n\in A_{\{i,j\}}$. 
	Let $\mu$ be the distribution of this random structure.

	Observe that
	the measure $\mu$ is invariant, since the random quantifier-free type of a tuple of distinct elements does not depend on the choice of tuple. Further, $\mu$ is ergodic by Theorem~\ref{thm:ergodic}, since the random quantifier-free types of disjoint tuples are independent.
	We call this ergodic structure $\mu$ the \defn{kaleidoscope random graph}. (Note that in \cite{AFNP}, this term is used instead to refer to models of its almost-sure first-order theory.)

	Note that there are continuum-many quantifier-free $2$-types consistent with $\Th_{\FO}(\mu)$,
	each of which is realized with probability $0$ in $\mu$. Any particular countable $L$-structure realizes at most countably many such types, and so $\mu$ assigns measure $0$ to its isomorphism class. Hence $\mu$ is properly ergodic.

In fact, for every $A\in 2^\n$, the theory $\Th(\mu)$ contains the sentence $$\lnot \exists x\, \exists y\, \Bigl(\bigdoublewedge_{n\in A} xR_n y\land \bigdoublewedge_{n\notin A} \lnot xR_n y\Bigr).$$
	Since all quantifier-free $2$-types consistent with $\Th_{\FO}(\mu)$
	are ruled out by $\Th(\mu)$, the theory $\Th(\mu)$ has no models of any cardinality. 
In fact, the 	
	complete $\Lwow$-theory  of any properly ergodic structure has no models of any cardinality, as shown in Corollary~\ref{cor:nomodels}.
	Note, however, that any countable fragment $F$ of $\Lwow$ only contains countably many of the sentences above, so $\Th_F(\mu)$ only rules out countably many of the quantifier-free $2$-types. 

	Restricting to the first-order fragment, the theory $\Th_{\FO}(\mu)$ has several nice properties. It is the model companion of the universal theory asserting that each $R_n$ is irreflexive and symmetric. It can be axiomatized by extension axioms, analogous to those in the theory of the Rado graph: in each finite sublanguage $L^*\subseteq L$, for every finite tuple $A$ and non-redundant quantifier-free $1$-type over $A$ in the language $L^*$ consistent with $\Th_{\FO}(\mu)$, there is some element $b$ satisfying that quantifier-free type.
	The reduct of $\Th_{\FO}(\mu)$ to any finite sublanguage is countably categorical, but $\Th_{\FO}(\mu)$ has continuum-many countable models (since there are continuum-many quantifier-free $2$-types consistent with $\Th_{\FO}(\mu)$).
	In fact, for all countable fragments $F$ and properly ergodic structures $\mu$, the theory $\Th_F(\mu)$ has continuum-many countable models,
	as we also show in Corollary~\ref{cor:nomodels}.

For arbitrary arity $k \ge 1$,
	an analogous construction
	produces the \defn{kaleidoscope random $k$-uniform hypergraph}.
	We call the case $k = 1$ the \defn{kaleidoscope random predicate}.
\end{example}

We now use the latter example to illustrate the distinction between group-theoretic and syntactic definable closure.

\begin{example}[The theory of the kaleidoscope random predicate]\label{ex:trivialdcl}
Let $T$ be the first-order theory of the kaleidoscope random predicate (see Example~\ref{ex:kaleidoscopes}) in the language $\{P_n\mid n\in \n\}$. The theory $T$ says that for every $m\in \n$ and every subset $A\subseteq [m]$, there is an element $x$ such that for all $n\in [m]$, the relation $P_n(x)$ holds if and only if $n\in A$.

Now let $T'$ be $T$ together with the infinitary sentence
\[
	\forall x\, \forall y\, (\bigdoublewedge_{n\in \n} (P_n(x)\leftrightarrow P_n(y))\rightarrow x = y).
\]

The kaleidoscope random predicate almost surely satisfies $T'$. Each of the continuum-many quantifier-free $1$-types is realized with probability $0$, and since the quantifier-free $1$-types of distinct elements of $\n$ are independent, almost surely no $1$-type is realized more than once.

In a model $M$ of $T'$, no two elements have the same quantifier-free $1$-type. Hence $\Aut(M)$ is the trivial group, and $M$ has non-trivial group-theoretic dcl. Let $F$ be the countable fragment of $\Lwow$ generated by $T'$. Then $F$ does not contain the conjunctions of the form $\doublewedge_{n\in A} P_n(x)\land \doublewedge_{n\notin A} \lnot P_n(x)$ for $A\subseteq \n$ needed to pin down elements uniquely. In fact, the complete $F$-theory of the kaleidoscope random predicate (which extends $T'$) has trivial dcl, by Theorem~\ref{thm:trivialdcl}.
\end{example}

We will see that the presence of a formula $\chi(\xx)$ of positive measure, such that every type containing $\chi$ has probability $0$ of being realized, is a characteristic feature of properly ergodic structures.

In the kaleidoscope random graph (Example~\ref{ex:kaleidoscopes}), $x\neq y$ is such a formula $\chi(x,y)$, since every non-redundant quantifier-free $2$-type is realized with probability $0$. In contrast to the kaleidoscope random graph, Example~\ref{ex:maxgraph} shows that there are properly ergodic structures in which these $0$-probability types have infinitely many realizations if they are realized at all. 

On the other hand, in Example~\ref{ex:nonexample}, we describe a transformation (known as the ``blow-up''), which when applied to the Kaleidoscope random predicate, leads to each of the continuum-many $1$-types being realized infinitely many times (if at all), and yet whose resulting theory has no properly ergodic models. This shows that merely having continuum-many types in a theory with trivial dcl does not imply the existence of a properly ergodic model of the theory. 

These phenomena motivate the definition of rootedness in Section~\ref{sec:rootedness}.

\begin{example}[The max random graph]\label{ex:maxgraph}
As in Example~\ref{ex:kaleidoscopes}, let $L = \{R_n\mid n\in \mathbb{N}\}$, where each $R_n$ is a binary relation symbol. We build a random $L$-structure with domain $\n$ such that the interpretation of each $R_n$ is irreflexive and symmetric. For each $i\in \n$, independently choose a random element $A_i\in 2^\n$. Now for each pair $\{i,j\}$, let $A_{ij} = \max(A_i,A_j)$, where we give $2^\n$ its lexicographic order. We set $iR_n j$ if and only if $n\in A_{ij}$. 

We have continuum-many quantifier-free $2$-types $\{p_A\mid A\in 2^\n\}$, where $xR_n y\in p_A$ if and only if $n\in A$, and each is realized with probability $0$, since if $(i,j)$ realizes $p_A$, we must have $A_i = A$ or $A_j = A$.

As long as $A_i$ is not the constant $0$ sequence (which appears with probability $0$), then for any $j\neq i$, there is a positive probability, conditioned on the choice of $A_i$, that $A_j \leq A_i$, and hence $\qftp(i,j) = p_{A_i}$. Since the $A_j$ are chosen independently, almost surely the event $A_j\leq A_i$ occurs for infinitely many $j$. So, almost surely, any non-redundant quantifier-free $2$-type that is realized is realized infinitely many times. However, since the probability that $A_i= A_{j}$ when $i\neq j$ is $0$, almost surely all realizations of $p_{A_i}$ have a common intersection, namely the vertex $i$.
\end{example}

We now describe a modification of the theory of the kaleidoscope random predicate so that all of its ergodic models are not properly ergodic. 

\begin{example}[The blow-up of the theory of the kaleidoscope random predicate]\label{ex:nonexample}
Let $L = \{E\}\cup \{P_n\mid n\in \n\}$, and let $T$ be the model companion of the universal theory asserting that $E$ is an equivalence relation and the $P_n$ are unary predicates respecting $E$ (if $xEy$, then $P_n(x)$ if and only if $P_n(y)$). This is similar to the first-order theory of the kaleidoscope random predicate, but with each element replaced by an infinite $E$-class.

There is no properly ergodic structure that satisfies $T$ almost surely. Indeed, suppose $\mu\modelsas T$. Then for every quantifier-free $1$-type $p$, there is some probability $\mu(p)$ that $p$ is the quantifier-free type of the element $i\in \n$, and, by invariance, $\mu(p)$ does not depend on the choice of $i$. We denote by $S^1_\qf(\mu)$ the (countable) set of quantifier-free $1$-types with positive measure. If $\sum_{p\in S^1_\qf(\mu)} \mu(p) = 1$, then almost surely only the types in $S^1_\qf(\mu)$ are realized, since $\mu\modelsas \forall x\, \bigdoublevee_{p\in S^1_\qf(\mu)}\bigdoublewedge_{\varphi\in p} \varphi(x)$. Further, $\mu$ determines, for each $p\in S^1_\qf(\mu)$, the number of $E$-classes on which $p$ is realized (among $\{1,2,\dots,\aleph_0\}$), since each of the countably many choices is expressible by a sentence of $\Lwow$. The data of which quantifier-free $1$-types are realized, and how many $E$-classes realize each, determines a unique countable $L$-structure up to isomorphism, so $\mu$ is not properly ergodic. 

On the other hand, if $\sum_{p\in S^1_\qf(\mu)} \mu(p) < 1$, then almost surely some types that are not in $S^1_\qf(\mu)$ are realized. Any such type $p$ is realized with probability $0$, and, by ergodicity, the quantifier-free $1$-types of distinct elements of $\n$ are independent. So, almost surely, each of the $0$-probability types is realized at most once. This contradicts the fact that any realized type must be realized on an entire infinite $E$-class.
\end{example}

The next example shows why it important to use $\Lwow$ when performing the Morley--Scott analysis.

\begin{example}[A kaleidoscope-like bipartite graph]\label{ex:infinitary}
Let $L = \{P\} \cup \{R^i_j\mid i,j\in \n\}$, where $P$ is a unary predicate and the $R^i_j$ are binary relations, 
and
let $T$ be the model companion of the following universal theory:
\begin{enumerate}
\item $\forall x\, \forall y\, (R^i_j(x,y)\rightarrow (P(x)\land \lnot P(y)))$ for all $i$ and $j$.
\item $\forall x\forall y\, \lnot (R^i_0(x,y)\land R^{i'}_0(x,y))$ for all $i\neq i'$.
\item $\forall x\, \forall y\, (R^i_{j+1}(x,y)\rightarrow R^i_j(x,y))$ for all $i$ and $j$.
\end{enumerate}

Thus, a model of $T$ is a bipartite graph in which each edge from $x$ to $y$ is labeled by some $i\in \n$ (in the superscript) and the set of all $j< k$ for some $k\in \n_+\cup\{\infty\}$ (in the subscript), where $\n_+$ denotes the positive natural numbers.

Now $T$ is a complete theory with quantifier elimination and with only countably many quantifier-free types over $\emptyset$. Hence, by countable additivity, if $\mu$ is an ergodic structure that satisfies $T$ almost surely, then there is no positive-measure first-order formula $\chi(\xx)$ such that every type containing $\chi$ has measure $0$. Nevertheless, we will describe a properly ergodic structure that almost surely satisfies $T$.

First, for each $x\in \n$, let $P(x)$ hold with independent probability $\onehalf$, and pick $A_x\in 2^{\n}$ independently at random. Now for each pair $x\neq y$, if $P(x)$ and $\lnot P(y)$, then we choose which of the $R^i_j$ will hold of $(x,y)$. First independently choose $i\in \n$ according to a geometric distribution where $i = n$ holds with probability $2^{-(n+1)}$. Then, if $i\in A_x$, independently choose $k\in \n_+\cup\{\infty\}$ according to a geometric distribution where $k = \infty$ holds with probability $\onehalf$ and $k = n$ holds with probability $2^{-(n+1)}$ for $n\in \n_+$. On the other hand, if $i\notin A_x$, then independently choose $k\in \n_+$ according to a geometric distribution where $k = n$ holds with probability $2^{-n}$. Finally, for this choice of $i$ and $k$, we let $R^i_j(x,y)$ hold for all $j<k$.

In the resulting random structure, we can almost surely recover $A_x$ from every $x\in P$, since if $i\in A_x$, then almost surely there is some $y$ such that $R^i_j(x,y)$ for all $j\in \n$ (that is, the choice $k = \infty$ was made for the pair $(x,y)$), whereas this outcome is impossible if $i\notin A_x$. Thus the structure encodes a countable set of elements of $2^\n$, each of which occurs with probability $0$. 

The information encoding $A_x$ is part of the $1$-type of $x$ in any countable fragment of $\Lwow$ containing the infinitary formulas $\{\exists y\, \doublewedge_{j \in \n} R^i_j(x,y)\mid i\in \n\}$, but it is not expressible in first-order logic.
\end{example}

With the exception of Example~\ref{ex:geometric} (and G\'abor Kun's example alluded to in \S\ref{intro-subsec-propergodic}), the preceding examples have all used infinite languages, as this is the easiest setting in which to split the measure over continuum-many types. We conclude with an elementary example in the language with a single binary relation, which encodes the kaleidoscope random predicate into a directed graph, in a way that we easily verify is properly ergodic.

\begin{example}[A directed graph encoding the kaleidoscope random predicate]\label{ex:finiteL}
Let $L = \{R\}$, where $R$ is a binary relation. In our probabilistic construction, we will enforce the following almost surely:
\begin{itemize}
\item Let $O = \{x\mid R(x,x)\}$, and $P = \{x\mid \lnot R(x,x)\}$. Then $O$ and $P$ are both infinite sets.
\item If $R(x,y)$, then either $x$ and $y$ are both in $O$, or $x$ is in $P$ and $y$ is in $O$.
\item $R$ is a preorder on $O$. Denote by $xEy$ the induced equivalence relation $R(x,y) \land R(y,x)$. Then $E$ has infinitely many infinite classes, and $R$ linearly orders the $E$-classes with order type $\omega$.
\item Given $x\in P$ and $y,z\in O$, if $R(x,y)$ and $yEz$, then $R(x,z)$. So $R$ relates each element of $P$ to some subset of the $E$-classes. 
\end{itemize}
Thus we can interpret the kaleidoscope random predicate on $P$, where the $n^\tth$ predicate $P_n$ holds of $x$ if and only if $x$ is $R$-related to the $n^\tth$ class in the linear order on $O$. 

Now it is straightforward to describe the probabilistic construction: for each $i\in \mathbb{N}$, independently let $R(i,i)$ hold with probability $\onehalf$. This determines whether $i$ is in $O$ or $P$. If $i\in O$, we choose which $E$-class to put $i$ in, under the order induced by $R$, selecting the $n^\tth$ class independently with probability $2^{-(n+1)}$. These choices determine all the $R$-relations between elements of $O$. On the other hand, if $i\in P$, we pick $A_i\in 2^\n$ independently at random and relate $i$ to each the $n^\tth$ class in $O$ if and only if $n\in A_i$.

This describes an ergodic structure $\mu$, since the quantifier-free types of disjoint tuples are independent. We obtain the properties described in the bullet points above almost surely, and since $\omega$ is rigid, any isomorphism between structures satisfying these properties must preserve the order on the $E$-classes. For any subset of the $E$-classes, the probability is $0$ that there is an element of $P$ which is related to exactly those $E$-classes, and so $\mu$ is properly ergodic.
\end{example}

\section{Morley--Scott analysis of ergodic structures}\label{sec:analysis}

Throughout this section, let $\mu$ be an ergodic structure. Recall from 
Remark~\ref{measure-of-type} that for a countable fragment $F$ of $\Lwow$ and an $F$-type $p$, the abbreviation $\theta_p(\xx)$ means $\doublewedge_{\varphi\in p}\varphi(\xx)$, and the notation $\mu(p)$ 
means $\mu(\theta_p(\xx))$ and is called the \emph{measure of $p$}.

\begin{definition}\label{def:postypes}
We denote by $S^n_F(\mu)$ the set $\{p\mid \mu(p)>0\}$ of \defn{positive-measure $F$-types} in the variables $x_0,\dots,x_{n-1}$. We include the case $n = 0$: $S^0_F(\mu)$ has one element, namely $\Th_F(\mu)$.
\end{definition}

\begin{lemma}\label{lem:ctbletypes}
For all $n\in \n$, we have $|S^n_F(\mu)|\leq \aleph_0$.
\end{lemma}
\begin{proof}
Fix a tuple $\aa$ of distinct elements from $\omega$. The sets $\{\extent{\theta_p(\aa)}\mid p\in S^n_F(\mu)\}$ are disjoint sets of positive measure in $\Str_L$. By additivity of $\mu$, for all $m\in \n$, $P_m = \{p\in S^n_F(\mu)\mid \mu(p)\geq \nicefrac{1}{m}\}$ is finite (of size at most $m$), so $S^n_F(\mu) = \bigcup_{m\in\omega}P_m$ is countable.
\end{proof}

We build a sequence $\{F_\alpha\}_{\alpha\in\omega_1}$ of countable fragments of $\Lwow$ of length $\omega_1$, depending on the ergodic structure $\mu$:
\begin{align*}
F_0 &= \mathrm{FO}, \text{the first-order fragment.}\\
F_{\alpha+1} &= \text{the fragment generated by }F_\alpha \cup \Bigl\{\theta_p(\xx)\mid p\in \bigcup_{n\in\n}S^n_{F_\alpha}(\mu)\Bigr\}.\\
F_\gamma &= \bigcup_{\alpha<\gamma}F_{\alpha}, \text{if $\gamma$ is a limit ordinal.}
\end{align*}

\begin{definition} \label{def:splits}
We say that $p\in S^n_{F_\alpha}(\mu)$ \defn{splits} at $\beta>\alpha$ if $\mu(q) < \mu(p)$ for all types $q\in S^n_{F_\beta}(\mu)$ such that $p\subseteq q$. We say that $p$ \defn{splits later} if there exists $\beta$ such that $p$ splits at $\beta$. We say that $\mu$ \defn{has stabilized} at $\gamma$ if for all $n\in \n$, no type in $S^n_{F_\gamma}(\mu)$ splits later.
\end{definition}

\begin{lemma}\label{lem:splitting}
Let $\alpha<\beta<\gamma$.
\begin{enumerate}
\item If a type $p\in S^n_{F_\alpha}(\mu)$ splits at $\beta$, then $p$ also splits at $\gamma$.
\item Suppose $p\in S^n_{F_\beta}(\mu)$ splits at $\gamma$. Then $p' = p\cap F_\alpha$ is in $S^n_{F_\alpha}(\mu)$ and also splits at $\gamma$.
\item If no type in $S^n_{F_\alpha}(\mu)$ splits later, then no type in $S^n_{F_\beta}(\mu)$ splits later.
\end{enumerate}
\end{lemma}
\begin{proof}
(1)
Pick $q\in S^n_{F_\gamma}(\mu)$ with $p\subseteq q$, and let $q' = q\cap F_\beta$. Then $\mu(q) \leq \mu(q') < \mu(p)$, since $p$ splits at $\beta$. 

\vspace*{0.5em}
\noindent (2)
First, $0< \mu(p)\leq \mu(p')$, so $p' \in S^n_{F_\alpha}(\mu)$. Pick $q\in S^n_{F_\gamma}(\mu)$ such that $p'\subseteq q$. If $p\subseteq q$, then $\mu(q)<\mu(p)\leq \mu(p')$, since $p$ splits at $\gamma$. And if $p\not\subseteq q$, then $\mu(q) \leq \mu(p')-\mu(p) <\mu(p')$, since $\mu(p)>0$. In either case, $\mu(q) < \mu(p')$, so $p'$ splits at $\gamma$.

\vspace*{0.5em}
\noindent (3)
If some type in $S^n_{F_\beta}(\mu)$ splits later, then by (2), $p' = p\cap F_\alpha$ also splits later, and $p'\in S^n_{F_\alpha}(\mu)$.
\end{proof}

\begin{lemma}\label{lem:stabilizing}
There is some countable ordinal $\gamma$ such that $\mu$ has stabilized at $\gamma$.
\end{lemma}
\begin{proof}
Fix $n\in \n$. For each $\alpha\in\omega_1$, let 
\begin{align*}
\Sp(\alpha) &= \{p\in S^n_{F_\alpha}(\mu)\mid p \text{ splits later}\},\\
r_\alpha &= \sup\{\mu(p)\mid p\in \Sp(\alpha)\}.
\end{align*}
Note that $\Sp(\alpha)$ is countable, since $S^n_{F_\alpha}$ is. If $\Sp(\alpha)$ is non-empty, then $r_\alpha > 0$, and in fact the supremum is achieved by finitely many types, since $\sum_{p\in \Sp(\alpha)} \mu(p) \leq 1.$

By Lemma~\ref{lem:splitting} (2), the measure of any type in $\Sp(\beta)$ is bounded above by the measure of a type in $\Sp(\alpha)$, namely its restriction to $F_\alpha$. So we have $r_\beta\leq r_\alpha$ whenever $\alpha <\beta$. 

Now assume for a contradiction that $\Sp(\alpha)$ is non-empty for all $\alpha$. We build a strictly increasing sequence $\langle \alpha_\delta\rangle_{\delta\in \omega_1}$ in $\omega_1$, such that $\langle r_{\alpha_\delta}\rangle_{\delta\in \omega_1}$ is a strictly decreasing sequence in $[0,1]$. Begin with $\alpha_0 = 0$.

At each successor stage, we are given $\alpha = \alpha_\delta$, and we seek $\beta = \alpha_{\delta+1}$ with $r_\beta < r_\alpha$. Since $\Sp(\alpha)$ is non-empty, there are finitely many types $p_1,\dots,p_n$ of maximal measure $r_\alpha>0$. For each $i$, pick $\beta_i>\alpha$ such that $p_i$ splits at $\beta_i$, and let $\beta = \max(\beta_1,\dots,\beta_n)$. By Lemma~\ref{lem:splitting} (1), each $p_i$ splits at $\beta$. Let $q$ be a type in $\Sp(\beta)$ with $\mu(q) = r_\beta$, and let $q' = q\cap F_\alpha$. By Lemma~\ref{lem:splitting} (2), $q'\in \Sp(\alpha)$. If $q'$ is one of the $p_i$, then $\mu(q) < \mu(p_i) = r_\alpha$, since $p_i$ splits at $\beta$. If not, then $\mu(q) \leq \mu(q') < r_\alpha$. In either case, $r_\beta = \mu(q) < r_\alpha$.

If $\lambda$ is a countable limit ordinal, let $\alpha_\lambda = \sup_{\delta<\lambda} \alpha_\delta$. This is an element of $\omega_1$, since $\omega_1$ is regular. And for all $\delta < \lambda$, since $\alpha_{\delta+1}<\alpha_\lambda$, we have $r_{\alpha_\lambda} \leq r_{\alpha_{\delta+1}} < r_{\alpha_\delta}$. 

Of course, there is no strictly decreasing sequence of real numbers of length $\omega_1$, since $\mathbb{R}$ contains a countable dense set. Hence there is some $\gamma_n\in\omega_1$ such that $\Sp(\gamma_n)$ is empty, i.e., no type in $S^n_{F_{\gamma_n}}$ splits later. Let $\gamma = \sup_{n\in\n} \gamma_n \in \omega_1$. Then by Lemma~\ref{lem:splitting} (3), $\mu$ has stabilized at $\gamma$.
\end{proof}

We can think of the minimal ordinal $\gamma$ such that $\mu$ has stabilized at $\gamma$ as an analogue of the Scott rank for the ergodic structure $\mu$. Since no $F_\gamma$-type splits later, every positive-measure $F_{\gamma+1}$-type $q$ is isolated by the $F_{\gamma+1}$-formula $\theta_p$ for its restriction $p = q\cap F_\gamma$, relative to $\Th_{F_{\gamma+1}}(\mu)$. Lemma~\ref{lem:concentrating} says that if every tuple satisfies one of these positive-measure types almost surely, then $\mu$ almost surely satisfies a Scott sentence.

\begin{lemma}\label{lem:concentrating} 
	Suppose that $\mu$ has stabilized at $\gamma$, and that for all $n\in \n$, 
	\[
		\sum_{p\in S^n_{F_\gamma}(\mu)}\mu(p) = 1.
	\]
	Then $\mu$ concentrates on a countable structure.
\end{lemma}

\begin{proof}
For each type $r(\xx)\in S^n_{F_\gamma}(\mu)$ (we include the case $n = 0$), let $E_r$ be the set of types $q(\xx,y)\in S^{n+1}_{F_\gamma}(\mu)$ with $r\subseteq q$. Fix a type $p(\xx)\in S^n_{F_\gamma}(\mu)$, let $\varphi_p$ be the sentence 
	\[
		\forall \xx\, \Bigl( \theta_p(\xx)\rightarrow \forall (y\notin \xx)\,\bigdoublevee_{q\in E_p}\theta_q(\xx,y) \Bigr),
	\]
and let $\psi_p$ be the sentence
\[
\forall \xx\,\Bigl( \theta_p(\xx) \rightarrow \bigdoublewedge_{q\in E_p} \exists (y\notin \xx)\,\theta_q(\xx,y)\Bigr)
\]
Here $\forall (y\notin \xx)\rho(\xx,y)$ and $\exists (y\notin \xx)\rho(\xx,y)$ are shorthand for $\forall y((\bigwedge_{i = 0}^{n-1}y\neq x_i)\rightarrow \rho(\xx,y))$ and $\exists y((\bigwedge_{i = 0}^{n-1}y\neq x_i)\land \rho(\xx,y))$, respectively. We would like to show that $\mu$ satisfies $\varphi_p$ and $\psi_p$ almost surely.

By assumption, and since every $q\in S^{n+1}_{F_\gamma}(\mu)$ is in $E_r$ for a unique $r\in S^n_{F_\gamma}(\mu)$,
\[
	1 = \sum_{q\in S^{n+1}_{F_\gamma}(\mu)}\mu(q)= \sum_{r\in S^n_{F_\gamma}(\mu)}\,\sum_{q\in E_r} \mu(q).
\]

Then for all $r\in S^n_{F_\gamma}(\mu)$, we must have $$\mu(r) = \sum_{q\in E_r}\mu(q).$$ 

In particular, this is true for $r = p$, so for any tuple $\aa$ and any $b$ not in $\aa$, $\bextent{\bigdoublevee_{q\in E_p} \theta_q(\aa,b)}$ has full measure in $\extent{\theta_p(\aa)}$ (this is true even when $\aa$ contains repeated elements, since in that case $\extent{\theta_p(\aa)}$ has measure $0$). A countable intersection (over $b\in \n\setminus \enum{\aa}$) of subsets of $\extent{\theta_p(\aa)}$ with full measure still has full measure, so 
\[
	 \mu\Bigl(\Bextent{\theta_p(\aa)\rightarrow \forall (y\notin \aa) \bigdoublevee_{q\in E_p}\theta_q(\aa,y)}\Bigr) = 1.
\]
Taking another countable intersection over all tuples $\aa$, we have $\mu\modelsas\varphi_p$.

We turn now to $\psi_p$. Since $\mu$ stabilizes at $\gamma$, there is a (necessarily unique) extension of $p$ to a type $p^*\in S^n_{F_{\gamma+1}}(\mu)$ with $\mu(p^*) = \mu(p)$. Let $q(\xx,y)$ be any type in $E_p$, and let $\upsilon_q(\xx)\in F_{\gamma+1}$ be the formula $(\exists y\notin \xx)\,\theta_q(\xx,y)$. Note that $\theta_q(\xx,y)$ implies $\upsilon_q(\xx)$ and $\upsilon_q(\xx)$ implies $\theta_p(\xx)$. So $\mu(\upsilon_q(\xx))\geq \mu(q) > 0$, and we must have $\upsilon_q(\xx)\in p^*$, otherwise $\mu(p^*)\leq \mu(p) - \mu(\upsilon_q(\xx))$. Finally, we conclude that for any tuple $\aa$, the set $\extent{\upsilon_q(\aa)}$ has full measure in $\extent{\theta_p(\aa)}$, since $\mu(p) = \mu(p^*) \leq \mu(\upsilon_q(\xx)) \leq \mu(p)$. 

As before, a countable intersection of subsets with full measure has full measure, so 
\[
\mu\Bigl(\Bextent{\theta_p(\aa)\rightarrow \bigdoublewedge_{q\in E_p}\exists(y\notin \aa)\,\theta_q(\aa,y)}\Bigr) = 1.
\]
Taking another countable intersection over all tuples $\aa$, we have $\mu\modelsas \psi_p$.

Let $T = \Th_{F_\gamma}(\mu)\cup \{\varphi_p,\psi_p\mid p\in \bigcup_{n\in \n} S^n_{F_\gamma}(\mu)\}$, and note that $T$ is countable. Since $\mu$ almost surely satisfies $T$, it suffices to show that any two countable models of $T$ are isomorphic. This is a straightforward back-and-forth argument, using $\varphi_p$ and $\psi_p$ to extend a partial $F_\gamma$-elementary isomorphism defined on a realization of $p$ by one step: $\varphi_p$ tells us that each one-point extension in one model realizes one of the types in $E_p$, and $\psi_p$ tells us that every type in $E_p$ is realized in a one-point extension in the other model. To start, the empty tuples in any two models of $T$ satisfy the same $F_\gamma$-type, namely $\Th_{F_\gamma}(\mu)$.
\end{proof}

\begin{theorem}\label{thm:characterization1}
Let $\mu$ be an ergodic structure. Then $\mu$ is properly ergodic if and only if for every countable fragment $F$ of $\Lwow$, there is a countable fragment $F' \supseteq F$ and a formula $\chi(\xx)$ in $F'$ such that $\mu(\chi(\xx)) > 0$, but $\mu(p) = 0$ for every $F'$-type $p(\xx)$ containing $\chi(\xx)$.
\end{theorem}
\begin{proof}
Suppose $\mu$ is properly ergodic. By Lemma~\ref{lem:stabilizing}, $\mu$ stabilizes at some $\gamma$, and by Lemma~\ref{lem:concentrating}, there is some $n$ such that $\sum_{p\in S^n_{F_\gamma}(\mu)} \mu(p) < 1$. Let $\chi(\xx)$ be the formula $\doublewedge_{p\in S^n_{F_\gamma}(\mu)} \lnot \theta_p(\xx)$. Then $\mu(\chi(\xx)) >0$.

Let $F'$ be the countable fragment generated by $F\cup F_\gamma\cup\{\chi(\xx)\}$, and suppose that $p(\xx)$ is an $F'$-type containing $\chi(\xx)$. Let $q = p \cap F_\gamma$. Then $q$ is an $F_\gamma$ type that is consistent with $\chi(\xx)$, so $q\notin S^n_{F_\gamma}(\mu)$, and $\mu(p)\leq \mu(q) = 0$.

Conversely, suppose we have such a fragment $F'$ and such a formula $\chi(\xx)$. Since $\mu(\chi(\xx))>0$, by ergodicity, $\mu\modelsas \exists\xx\, \chi(\xx)$. Let $M$ be a countable structure. If $M$ contains no tuple satisfying $\chi$, then $\mu$ assigns measure $0$ to the isomorphism class of $M$. On the other hand, if $M$ contains a tuple $\aa$ satisfying $\chi(\xx)$, then since $\mu$ assigns measure $0$ to the set of structures realizing $\tp_{F'}(\aa)$, it also assigns measure $0$ to the isomorphism class of $M$. So $\mu$ is properly ergodic.
\end{proof}

By countable additivity, if a sentence $\varphi$ of $\Lwow$ has only countably many countable models up to isomorphism, then any ergodic structure $\mu$ that almost surely satisfies $\varphi$ is almost surely isomorphic to one of its models. That is, no ergodic model of $\varphi$ is properly ergodic. We show now that the same is true if $\varphi$ is a counterexample to Vaught's conjecture, i.e., a sentence with uncountably many, but fewer than continuum-many, countable models.

\begin{corollary}[``Vaught's Conjecture for ergodic structures'']\label{cor:continuum} Let $\varphi$ be a sentence of $\Lwow$. If there is a properly ergodic structure $\mu$ such that $\mu\modelsas \varphi$, then $\varphi$ has continuum-many countable models up to isomorphism.
\end{corollary}

\begin{proof}
This is a consequence of Theorem~\ref{thm:characterization1} and an observation due to Morley \cite{Morley}: for any countable fragment $F$ of $\Lwow$ containing $\varphi$ and any $n\in\n$, the set $S^n_F(\varphi)$ of $F$-types consistent with $\varphi$ is an analytic subset of $2^F$. Since analytic sets have the Perfect Set Property, if $|S^n_F(\varphi)|>\aleph_0$, then $|S^n_F(\varphi)| = 2^{\aleph_0}$. And since a countable structure realizes only countably many $n$-types, if $|S^n_F(\varphi)| = 2^{\aleph_0}$, then $\varphi$ must have continuum-many countable models up to isomorphism.

Now let $\mu$ be the given properly ergodic structure, let $F$ be a countable fragment containing $\varphi$, let $F'$ and $\chi(\xx)$ be as in Theorem~\ref{thm:characterization1}, let $n$ be the length of the tuple $\xx$, and suppose for a contradiction that $|S^n_{F'}(\varphi)| \leq \aleph_0$. Let $U_\chi = \{p\in S^n_{F'}(\varphi)\mid \chi(\xx)\in p\}$. Then $U_\chi$ is countable, and, by our choice of $\chi(\xx)$, we have $\mu(p) = 0$ for all $p\in U_\chi$. Since $\mu(\extent{\varphi}) = 1$, for any tuple $\aa$ of distinct elements of $\n$, we have 
\[
	0 < \mu(\extent{\chi(\aa)}) = \mu(\extent{(\varphi\land \chi)(\aa)}) = \mu\Bigl(\bigcup_{p\in U_\chi} \extent{\theta_p(\aa)}\Bigr) = \sum_{p\in U_\chi} \mu(p),
\]
which is a contradiction, by countable additivity of $\mu$.
\end{proof}

Kechris has observed (in private communication) that Corollary~\ref{cor:continuum} also follows from a result in descriptive set theory \cite[Exercise~17.14]{Kechris}: an analogue for measure of a result of Kuratowski about category \cite{MR0461467}. However, our proof above provides additional model-theoretic information about properly ergodic structures.

Recall that $\Th(\mu)$ is the complete $\Lwow$-theory of $\mu$. As noted in Remark~\ref{rem:negscott}, $\mu$ is properly ergodic if and only if $\Th(\mu)$ has no countable models. In fact, if $\mu$ is properly ergodic, then $\Th(\mu)$ has no models at all. This is stronger, since the L\"{o}wenheim--Skolem theorem fails for complete theories of $\Lwow$.

\begin{corollary}\label{cor:nomodels}
	If $\mu$ is properly ergodic, then $\Th(\mu)$ has no models (of any cardinality). However, for any countable fragment $F$ of $\Lwow$, the theory $\Th_F(\mu)$ has continuum-many countable models up to isomorphism.
\end{corollary}
\begin{proof} 
Starting with any countable fragment $F$ (e.g., $F = \mathrm{FO}$), let $F'$ and $\chi(\xx)$ be as in Theorem~\ref{thm:characterization1}. Then $\mu(\chi(\xx))>0$, so by ergodicity, $\exists\xx\, \chi(\xx)\in \Th(\mu)$. Now if $\Th(\mu)$ has a model $M$, then there is some tuple $\aa$ from $M$ satisfying $\chi(\xx)$. Let $p$ be the $F'$-type of $\aa$. Since $p$ contains $\chi(\xx)$, we have $\mu(p) = 0$, and so $\lnot \exists \xx\, \theta_p(\xx)\in \Th(\mu)$, a contradiction.

The last assertion follows from Corollary~\ref{cor:continuum}, taking $\varphi = \doublewedge_{\psi\in \Th_{F}(\mu)} \psi$. 
\end{proof}

Corollary \ref{cor:nomodels} describes a general version of two phenomena observed for the Kaleidoscope random graph in Example \ref{ex:kaleidoscopes}.

\section{Rooted models}\label{sec:rootedness}

The Morley--Scott analysis in Section~\ref{sec:analysis} showed that proper ergodicity of $\mu$ can always be explained by a positive-measure formula $\chi(\xx)$ such that any type containing $\chi(\xx)$ has measure $0$. In a countable structure sampled from $\mu$, each of these types of measure $0$ will be realized ``rarely''. Sometimes ``rarely'' means ``at most once'', as in Examples~\ref{ex:kaleidoscopes} and~\ref{ex:trivialdcl}. But in Example~\ref{ex:maxgraph}, the max graph, we saw that a type $p$ of measure $0$ can be realized by infinitely many tuples, all of which share a common element $i\in\n$. In that example, if some element $A\in 2^{\n}$ is randomly selected at a vertex $i$, then for any other vertex $j$, there is a positive probability that $\qftp(i,j)$ is the type $p_A$ encoding $A$. In other words, the fact that the type $p_A$ is realized infinitely many times is explained by the fact that $p_A$ has positive measure, after the random choice of $A$ ``living at'' the vertex $i$. In this section, we will use the Aldous--Hoover--Kallenberg theorem from \S\ref{sec:aldous-hoover} to show that this behavior is typical.

Throughout this section, let $F$ be a countable fragment of $\Lwow$, and let $T$ be an $F$-theory. We write $S^{[n]}_F(T)$ for the subspace of $S^n_F(T)$ consisting of non-redundant $F$-types on $x_0, \dots, x_{n-1}$, i.e., those which contain $x_i\neq x_j$ for all $i\neq j$.

\begin{definition}\label{def:roots}
Let $p\in S^{[n]}_F(T)$ be a type realized in $M\models T$. An element $a\in M$ is called a \defn{root} of $p$ in $M$ if $a$ is an element of every tuple realizing $p$ in $M$. We use the same terminology for quantifier-free types in $S^{[n]}_\qf(T)$.
\end{definition}

\begin{remark}\label{rem:uniquerealization}
If a type $p$ has a unique realization in $M$, then $p$ has a root in $M$ (take any element of the unique tuple realizing $p$). When $n = 1$, the converse is true: a realized type $p(x)\in S^{[1]}_F(T)$ (or $S^{[1]}_\qf(T)$) has a root in $M$ if and only if it has a unique realization in $M$.
\end{remark}

\begin{definition}\label{def:rootedness}
Let $\chi(\xx)$ is a formula in $F$ such that $\chi(\xx)\rightarrow (\bigwedge_{i \neq j} x_i\neq x_j)\in T$. Then a model $M\models T$ is \defn{$\chi$-rooted} if every type $p(\xx)\in S^{[n]}_F(T)$ which contains $\chi$ and is realized in $M$ has a root in $M$. Again, we use the same terminology for quantifier-free formulas and types.
\end{definition}

\begin{remark}\label{rem:rootedborel}
We note that the property of $\chi$-rootedness is expressible by a sentence of $\Lwow$, although not necessarily a sentence of $F$, which asserts that for every tuple $\aa$ of distinct elements satisfying $\chi(\xx)$, there is some element $a_i$ of the tuple such that every other tuple $\bb$ with the same $F$-type as $\aa$ contains $a_i$. Hence the set of $\chi$-rooted models of $T$ is a Borel set in $\Str_L$.
\end{remark}

Our goal is to prove the following theorem.

\begin{theorem}\label{thm:rootedness}
Let $\mu$ be a properly ergodic structure, $F$ a countable fragment of $\Lwow$, and $\chi(\xx)$ a formula in $F$ such that $\mu(\chi(\xx)) > 0$ and $\chi(\xx)\rightarrow (\bigwedge_{i \neq j} x_i\neq x_j) \in \Th_F(\mu)$. Suppose that $\mu(p) = 0$ for every $F$-type $p$ containing $\chi(\xx)$. Then $\mu$ assigns measure $1$ to the set of $\chi$-rooted models of $\Th_F(\mu)$.
\end{theorem}

The idea of the proof is as follows: We take an AHK representation of $\mu$, sampling from which involves a family of i.i.d.\ random variables $(\xi_A)_{A\in\mathcal{P}_\fin(\n)}$. For a set $B$ with $0\leq |B|\leq n$, we say that an $n$-type $p$ is \emph{likely given} $\widehat{\xi}_B$ if after conditioning on the random variables $(\xi_A)_{A\in \mathcal{P}(B)}$, the type $p$ has a positive probability of being realized on a tuple containing all the elements of $B$ (see Definition~\ref{likelygiven-def} below).

Now for $C\subseteq B$, it happens with probability $0$ that a particular type $p$ jumps from being not likely given $\widehat{\xi}_C$ to being likely given $\widehat{\xi}_B$. As a consequence, it is almost surely the case that for every type $p$, the family of sets $N(p) = \{C\mid p \text{ is likely given }\widehat{\xi}_C\}$ is closed under intersection: given sets $A$ and $B$, the probability that the same type jumps from being not likely given $A\cap B$ to being likely given both $A$ and $B$ is $0$, since $\widehat{\xi}_A$ and $\widehat{\xi}_B$ are conditionally independent over $\widehat{\xi}_{A\cap B}$. 

Now for any type $p$ containing $\chi$, the family $N(p)$ contains all the sets on which $p$ is realized, and it does not contain $\emptyset$ (since $p$ has measure $0$, and, by ergodicity, the random variable $\xi_\emptyset$ is irrelevant) so the intersection of all sets on which $p$ is realized is almost surely nonempty, i.e., if $p$ is realized, then it almost surely has a root.

Unfortunately, the need to handle all continuum-many types uniformly introduces some technical complications in formalizing this intuitive argument. We will now tackle those technicalities.

Suppose $(f_n)_{n\in\n}$ is an AHK representation of an invariant measure $\mu$. It is a consequence of Lusin's theorem on measurable functions~\cite[Theorem 17.12]{Kechris} that every measurable function differs from a Borel function on a set of measure $0$. If $g_n$ agrees with $f_n$ almost everywhere, we may replace $f_n$ by $g_n$ in the AHK system and obtain another AHK representation of $\mu$. Hence, we may assume that each $f_n$ is Borel measurable.

We adopt the notation of \S\ref{sec:aldous-hoover} for the random variables $(\xi_A)_{A\in \mathcal{P}_{\fin}(\n)}$: for a tuple $\bb$, we will write $\widehat{\xi}_{\bb}$ to denote the family of random variables $(\xi_A)_{A\subseteq \enum{\bb}}$. Similarly, we will write $\widehat{x}_B$ as a shorthand for a family of values $(x_A)_{A\in \mathcal{P}(B)}\in [0,1]^{\mathcal{P}(B)}$. If $C\subseteq B$, we separate the family $\widehat{x}_B$ into $\widehat{x}_C = (x_A)_{A\in \mathcal{P}(C)}$ and $\widehat{x}_{B/C} = (x_A)_{A\in \mathcal{P}(B)\setminus \mathcal{P}(C)}$. If $B$ is enumerated as a tuple $\bb$, we will write $\widehat{x}_{\bb}$.

\begin{definition}
Let $p\in S^{[n]}_{\qf}(L)$, and let $B\in \mathcal{P}_{\fin}(\n)$, with $|B| = n$. Fix values $\widehat{x}_B\in [0,1]^{\mathcal{P}(B)}$. We say that $p$ \defn{is realized given $\widehat{x}_B$} if there is some enumeration of $B$ as a tuple $\bb$ such that $p = f_n(\widehat{x}_{\bb})$.
\end{definition}

While the particular type $f_n(\widehat{x}_{\bb})$ depends on the order in which $B$ is enumerated as a tuple, the set of types which are realized given $\widehat{x}_B$ does not depend on the order. Since there are $n!$ ways of enumerating $B$ as a tuple, at most $n!$ quantifier-free types are realized given $\widehat{x}_B$. Recall that part of the definition of an AHK system is that these types almost surely form an orbit under the action of $\sympar{n}$ on $S^{[n]}_\qf(L)$ by permuting variables. 

The set $R(B) = \{(p,\widehat{x}_B)\mid p\text{ is realized given }\widehat{x}_B\}\subseteq S^{[n]}_\qf(L)\times [0,1]^{\mathcal{P}(B)}$ is Borel. Indeed, $$R(B) = \bigcup_{\bb\text{ enumerating } B}\{(f_n(\widehat{x}_{\bb}),\widehat{x}_B)\mid \widehat{x}_B\in [0,1]^{\mathcal{P}(B)}\},$$
and the graph of a Borel function is a Borel set.

Let $C\in \mathcal{P}_{\fin}(\n)$, with $0\leq |C| \leq n$, and pick some $C\subseteq B$ with $|B| = n$. Let $1_{R(B)}$ be the indicator function of the event $R(B)$. The Fubini--Tonelli theorem for Borel measurable functions~\cite[Theorem 1.7.15]{Tao} tells us that for all $p\in S^{[n]}_\qf(L)$ and $\widehat{x}_C\in [0,1]^{\mathcal{P}(C)}$, the integral 
$$\int_{\widehat{x}_{B/C}\in [0,1]^{\mathcal{P}(B)\setminus\mathcal{P}(C)}}1_{R(B)}(p,\widehat{x}_C,\widehat{x_{B/C}})\,d\lambda_0^{\mathcal{P}(B)\setminus\mathcal{P}(C)}$$ 
is defined (here $\lambda_0$ is the Lebesgue measure on $[0,1]$, restricted to the Borel $\sigma$-algebra), and that the function 
$$P_C\colon (p,\widehat{x}_C)\mapsto \int_{\widehat{x}_{B/C}\in [0,1]^{\mathcal{P}(B)\setminus\mathcal{P}(C)}}1_{R(B)}(p,\widehat{x}_C,\widehat{x}_{B/C})\,d\lambda_0^{\mathcal{P}(B)\setminus\mathcal{P}(C)}$$
is Borel measurable. Abusing terminology somewhat, we call $P_C(p,\widehat{x}_C)$ the \defn{probability of $p$ given $\widehat{x}_C$}.

Observe that the definition of $P_C$ is independent of the choice of $B$, since if $C\subseteq B'$ and $f\colon B\to B'$ is a bijection fixing $C$, the induced map $S^{[n]}_\qf(L)\times [0,1]^{\mathcal{P}(B)}\to S^{[n]}_\qf(L)\times [0,1]^{\mathcal{P}(B')}$ carries $R(p,B)$ to $R(p,B')$.

\begin{definition}
	\label{likelygiven-def}
With notation as above, we say that $p$ is \defn{likely given $\widehat{x}_C$} if the probability of $p$ given $\widehat{x}_C$ is positive. For fixed values $\widehat{x}_C$, we denote by $S(\widehat{x}_C)$ the set of non-redundant quantifier-free $n$-types which are likely given $\widehat{x}_C$.
\end{definition}

We'd like to show that $S(\widehat{x}_C)$ is always countable. We'll need the following basic measure theory lemma.

\begin{lemma}\label{lem:events}
Let $(\Omega,\mathcal{F},\nu)$ be a probability space, and let $(E_i)_{i\in I}$ be an uncountable family of events, each of positive measure. Then there is some $x\in \Omega$ such that $x$ is in infinitely many of the $E_i$.
\end{lemma}
\begin{proof}
Since there are uncountably many events in the family, there is some $\varepsilon>0$ such that infinitely many have measure at least $\varepsilon$. Let $\{E_{i_n}\mid n\in \n\}$ be a countable sequence with $\nu(E_{i_n})\geq \varepsilon$ for all $n$. Then define $E'_N = \bigcup_{n\geq N} E_{i_n}$, for $n\in N$. We have $\nu(E'_N)\geq \varepsilon$. By continuity, $\nu(\bigcap_{N\in \n} E'_N) \geq \varepsilon$, so there is some $x\in \bigcap_{N\in \n} E'_N$. This $x$ is in infinitely many of the $E_{i_n}$. 
\end{proof}

\begin{lemma}\label{lem:rooted1} For any set $C$ with $0\leq |C|\leq n$, and any $\widehat{x}_C\in [0,1]^{\mathcal{P}(C)}$, the set $S(\widehat{x}_C)$ is countable.
\end{lemma}
\begin{proof}
Pick some $C\subseteq B$ with $|B| = n$. A type $p$ is likely given $\widehat{x}_C$ if and only if the event $E_p = \{\widehat{x}_{B/C}\mid (p,\widehat{x}_C,\widehat{x}_{B/C})\in R(B)\}\subseteq [0,1]^{\mathcal{P}(B)\setminus \mathcal{P}(C)}$ has positive measure. Since any point $\widehat{x}_{B/C}$ is in at most $n!$ of the events $E_p$ (corresponding to the types $f_n(\widehat{x}_{\bb})$ for the $n!$ enumerations of $B$ as a tuple), by Lemma~\ref{lem:events}, the set $S(\widehat{x}_C) = \{p\mid E_p\text{ has positive measure}\}$ is countable.
\end{proof}

\begin{lemma}\label{lem:rooted2}
For sets $D\subseteq C$ with $0\leq |D|\leq |C|\leq n$, and any $\widehat{x}_D\in [0,1]^{\mathcal{P}(D)}$, if $p$ is not likely given $\widehat{x}_D$, then the set $\{\widehat{x}_{C/D} \mid  p \text{ is likely given }\widehat{x}_C\}$ has measure $0$ in $[0,1]^{\mathcal{P}(C)\setminus\mathcal{P}(D)}$. 
\end{lemma}
\begin{proof} This is just the Fubini--Tonelli theorem. Pick some $C\subseteq B$ with $|B| = n$ (recall that $p$ is an $n$-type). Since $p$ is not likely given $\widehat{x}_D$, we have 
\begin{align*}
0 &= \int_{\widehat{x}_{B/D}}1_{R(B)}(p,\widehat{x}_D,\widehat{x}_{B/D})\,d\lambda_0^{\mathcal{P}(B)\setminus\mathcal{P}(D)}\\
&= \int_{\widehat{x}_{C/D}}\Bigl(\int_{\widehat{x}_{B/C}}1_{R(B)}(p,\widehat{x}_D,\widehat{x}_{C/D},\widehat{x}_{B/C})\,d\lambda_0^{\mathcal{P}(C)\setminus\mathcal{P}(D)}\Bigr)d\lambda_0^{\mathcal{P}(B)\setminus\mathcal{P}(C)}.
\end{align*}
And the interior integral, which is $0$ for almost all values of $\widehat{x}_{C/D}$, is the probability of $p$ given $\widehat{x}_C$.
\end{proof}

\begin{lemma}\label{lem:rooted3}
Fix sets $A,B\in \mathcal{P}_\fin(\n)$, with $0\leq |A|\leq n$ and $0\leq |B|\leq n$, and let $C = A\cap B$. Let $\widehat{\xi}_C$, $\widehat{\xi}_A = (\widehat{\xi}_C,\widehat{\xi}_{A/C})$ and $\widehat{\xi}_B = (\widehat{\xi}_C,\widehat{\xi}_{B/C})$ be our random variables on these sets. Almost surely, every quantifier-free type which is likely given $\widehat{\xi}_A$ and likely given $\widehat{\xi}_B$ is likely given $\widehat{\xi}_C$. That is, $S(\widehat{\xi}_C)\subseteq S(\widehat{\xi}_A)\cap S(\widehat{\xi}_B)$.
\end{lemma}
\begin{proof}
Observe that for any set $D$, the set $\{(p,\widehat{x}_D)\mid p\text{ is likely given }\widehat{x}_D\}\subseteq S^{[n]}_\qf(L)\times [0,1]^{\mathcal{P}(D)}$ is Borel. Indeed, it is the preimage of $(0,1]$ under the Borel function $P_D$. It follows that the set $$\{(p,\widehat{x}_C,\widehat{x}_{A/C},\widehat{x}_{B/C})\mid p\text{ is likely given $\widehat{x}_A$ and $\widehat{x}_B$ but not $\widehat{x}_C$}\}$$ is a Borel subset of $S^{[n]}_\qf(L)\times [0,1]^{\mathcal{P}(C)}\times [0,1]^{\mathcal{P}(A)\setminus \mathcal{P}(C)}\times [0,1]^{\mathcal{P}(B)\setminus \mathcal{P}(C)}$. Projecting out the first coordinate, we see that the set $$X^{A,B} = \{(\widehat{x}_C,\widehat{x}_{A/C},\widehat{x}_{B/C})\mid \text{some $p$ is likely given $\widehat{x}_A$ and $\widehat{x}_B$ but not $\widehat{x}_C$}\}$$
is analytic, and hence measurable (see \cite[Theorem 21.10]{Kechris}). Since our random variables are i.i.d.\ Lebesgue on $[0,1]$, we would like to show that $X^{A,B}$ has measure $0$ with respect to the Lebesgue measure on $[0,1]^{\mathcal{P}(C)}\times [0,1]^{\mathcal{P}(A)\setminus \mathcal{P}(C)}\times [0,1]^{\mathcal{P}(B)\setminus \mathcal{P}(C)}$. The knowledge that $X^{A,B}$ is measurable enables us to analyze its measure fiber-wise, using the Fubini--Tonelli theorem (the version for complete measures this time, \cite[Theorem 1.7.18]{Tao}).

So consider the fiber $X^{A,B}_{\widehat{x}_A}$ over $\widehat{x}_A = (\widehat{x}_C,\widehat{x}_{A/C})\in [0,1]^{\mathcal{P}(C)}\times [0,1]^{\mathcal{P}(A)\setminus \mathcal{P}(C)}$. 
\begin{align*}
X^{A,B}_{\widehat{x}_A} &= \{\widehat{x}_{B/C}\mid \text{some $p$ is likely given $\widehat{x}_A$ and $\widehat{x}_B$ but not $\widehat{x}_{C}$}\}\\
&= \bigcup_{p\in S(\widehat{x}_A)\setminus S(\widehat{x}_{C})} \{\widehat{x}_{B/C}\mid p\text{ is likely given $\widehat{x}_B$}\}
\end{align*}

By Lemma~\ref{lem:rooted1}, this is a countable union, and by Lemma~\ref{lem:rooted2}, each set in the union has measure $0$, so $X^{A,B}_{\widehat{x}_A}$ has measure $0$, and, by Fubini--Tonelli, $X^{A,B}$ has measure $0$.
\end{proof}

\begin{proof}[Proof of Theorem~\ref{thm:rootedness}] 
We have a properly ergodic structure $\mu$, a countable fragment $F$ of $\Lwow$, and a distinguished formula $\chi(\xx)$ in $F$. Let $n$ be the length of the tuple $\xx$. Let $L'$, $T'$, and $Q$ be the language, $\Pi_2^-$ theory, and countable set of partial quantifier-free types, respectively obtained from Theorem~\ref{thm:pithy} for the fragment $F$ and empty theory $T$. By Corollary~\ref{cor:measurebijection}, $\mu$ corresponds to an ergodic $L'$-structure $\mu'$, concentrated on those models of $T'$ that omit all the types in $Q$. 

For such models, each formula $\varphi(\yy)$ in $F$ is equivalent to the atomic $L'$-formula $R_\varphi(\yy)$, so we have $\mu'(R_\chi(\xx))>0$, and for every quantifier-free type $q$ containing $R_\chi(\xx)$, we have $\mu'(q) = 0$. It suffices to show that an $L'$-structure $\fM$ sampled from $\mu'$ is almost surely $R_\chi(\xx)$-rooted with respect to quantifier-free types. 

By Theorem~\ref{thm:aldous-hoover}, $\mu'$ has an AHK representation $(f_m)_{m\in \n}$. And since $\mu'$ is ergodic, by Theorem~\ref{thm:ergodic}, we can pick the functions $f_m$ so they do not depend on the argument indexed by $\emptyset$. As a consequence, $S(\widehat{\xi}_\emptyset) = \{p\mid \mu'(p)>0\}$. Indeed, the probability of a type $p$ given $\widehat{\xi}_\emptyset$ is obtained by integrating out all of the variables except $\xi_\emptyset$, which is irrelevant to $f_n$, so it is simply the probability that $p$ is realized on an arbitrary set of size $n$.

On the other hand, if $|B| = n$, then $S(\widehat{\xi}_B)$ is almost surely equal to the set of quantifier-free types realized on the $n!$ tuples $\bb$ enumerating $B$. Indeed, the probability of a type $p$ given $\widehat{\xi}_B$ is simply the indicator function $1_{R(B)}$ (no variables are integrated out).

Now for any tuple $\bb$, letting $B = \enum{\bb}$, if $\bb$ satisfies a type $p$ containing $R_\chi$, then almost surely the family of sets $C$ such that $p\in S(\widehat{\xi}_C)$ contains $B$, does not contain $\emptyset$, and is closed under intersection (by Lemma~\ref{lem:rooted3}). In particular, the intersection of this family is non-empty. Therefore, almost surely, $\fM$ is $\chi$-rooted.
\end{proof}

We conclude this section with a discussion of the tension between rootedness and trivial definable closure.

Let $M$ be a $\chi$-rooted model of an $F$-theory $T$. Suppose that $p(\xx)\in S^{[n]}_F(T)$ contains $\chi(\xx)$ and is realized in $M$, and let $a$ be a root of $p$ in $M$. If $M\models p(a,\bb)$, then $a$ is the unique element of $M$ satisfying $p(x,\bb)$, since if $c\neq a$, then $a$ is not in $c\bb$, so $c\bb$ does not realize $p$. This implies that $M$ has non-trivial group-theoretic definable closure, since every automorphism of $M$ fixing $\bb$ also fixes $a$. Note that $T$ may still have trivial definable closure, since $p$ is an $F$-type and, in general, is not equivalent to a formula in $F$.

We can conclude, however, that if an $F$-theory $T$ with trivial definable closure has a $\chi$-rooted model, then no non-redundant type that contains $\chi$ is isolated. Thus isolated types are not dense in $S^n_F(T)$. By standard facts about model theory in countable fragments of $\Lwow$ (see \cite{MR2062240}), this implies that $T$ does not have a prime model with respect to $F$-elementary embeddings, and that there are continuum-many types in $S^n_F(T)$ containing $\chi(\xx)$.

Theorem~\ref{thm:rootedness} implies that the theory of every properly ergodic structure $\mu$ exhibits this behavior: using Theorem~\ref{thm:characterization1} to obtain a countable fragment $F$ and an $F$-formula $\chi(\xx)$ of positive measure such that every type containing $\chi$ has measure $0$, the theory $\Th_F(\mu)$ has many $\chi$-rooted models and (by Theorem~\ref{thm:trivialdcl}) trivial dcl. 

Of course, given any particular $F$-type $p$ containing $\chi(\xx)$, we can try to bring $\chi$-rootedness into direct conflict with trivial dcl by moving to a larger countable fragment $F'$ which contains the formula $\theta_p(\xx)\defas \doublewedge_{\varphi\in p} \varphi(\xx)$ isolating $p$. But since $p$ has measure $0$, the theory $\Th_{F'}(\mu)$ contains the sentence $\forall\xx\,\lnot\theta_p(\xx)$, ruling out troublesome realizations of $p$. 

Of course, a countable fragment $F'$ can only isolate and rule out countably many of the continuum-many types of measure $0$ containing $\chi(\xx)$. For example, given the kaleidoscope random graph (Example~\ref{ex:kaleidoscopes}), we could extend from the first-order fragment to a countable fragment $F$ of $\Lwow$ containing some of the conjunctions $\doublewedge_{n\in A} xR_n y \wedge \doublewedge_{n\notin A} \lnot xR_n y$, for $A\in 2^\n$. Then the theory $\Th_{F'}(\mu)$ is essentially the same as $\Th_{\FO}(\mu)$, but with countably many of the continuum-many quantifier-free $2$-types forbidden.

\section{Constructing properly ergodic structures}\label{sec:construction}

In this section, given an $F$-theory $T$ having trivial dcl, we will use a single $\chi$-rooted model $M$ of $T$ to construct a properly ergodic model of $T$. The strategy is to build a Borel structure $\bbM$ equipped with a probability measure $\nu$, via an inverse limit of finite probability spaces. We use $M$ as a guide in the construction to ensure that $\bbM$ is also $\chi$-rooted. Then our ergodic structure $\mu$ will be obtained by i.i.d.\ sampling of countably many points from $\bbM$ according to $\nu$ and taking the induced substructure.

Having built the Borel structure $\bbM$, we proceed to rescale $\nu$, using a technique from~\cite{AFKP}, to obtain not just one but continuum-many properly ergodic structures concentrating on $T$.

\begin{definition}\label{def:borelstructure}
A \defn{Borel structure} $\bbM$ is an $L$-structure whose domain is a standard Borel space such that for every relation symbol $R$ of arity $\ar(R)$ in $L$, the subset $R\subseteq M^{\ar(R)}$ is Borel. A \defn{measured structure} is a Borel structure $\bbM$ equipped with an atomless probability measure $\nu$.
\end{definition}

Given a measured structure $(\bbM, \nu)$, there is a canonical measure $\mu_{\bbM,\nu}$ on $\Str_L$, obtained by sampling a countable $\nu$-i.i.d.\ sequence (of almost surely distinct points) from $\bbM$ and taking the induced substructure.
Somewhat more formally, $\mu_{\bbM,\nu}$ is the distribution of a random structure in $\Str_L$ whose atomic diagram on $\n$ is given by that of the random substructure of $\bbM$ with underlying set $\{a_i\mid i\in \n\}$, where $(a_i)_{i\in\Nats}$ is a $\nu$-i.i.d.\ sequence of (almost surely unique) elements in $\bbM$.

We now describe an AHK representation of the measures $\mu_{\bbM,\nu}$ in the sense of \S\ref{sec:aldous-hoover}.
Choose a measure-preserving Borel isomorphism $h$ from $[0,1]$ equipped with the uniform measure to the domain of $\bbM$ equipped with $\nu$, and 
for each $n$ let $\star_n$ be an arbitrary element of $S^{[n]}_\qf(L)$.
Then define functions $f_n\colon [0,1]^{\mathcal{P}_{\fin}([n])}\to S^{[n]}_\qf(L)$ by 
\[
	f_n\bigl((\xi_A)_{A\subseteq [n]}\bigr) = 
\begin{cases}
    \qftp(h({\xi_{\{0\}}}),\dots,h({\xi_{\{n-1\}}})) & \text{if } \xi_{\{i\}} \neq \xi_{\{j\}} \text{ for }i < j \in [n];\\
\star_n& \text{otherwise.}
\end{cases}
\]

Informally, these functions ignore the random variables $\xi_A$ when $|A|\neq 1$ and view the $(\xi_{\{a\}})_{a\in \n}$ as independent random variables with distribution $\nu$ taking their values in $\bbM$.

Now $(f_n)_{n\in \n}$ is an AHK system, so it induces an invariant measure 
on $\Str_L$. This measure is clearly the same as $\mu_{\bbM,\nu}$ described above via sampling of a random substructure. Since the $f_n$ do not depend on the argument indexed by $\emptyset$, the measure $\mu_{\bbM,\nu}$ is ergodic (Theorem~\ref{thm:ergodic}), which establishes the following lemma.

\begin{lemma}\label{lem:borelsampling}
Given a measured structure $(\bbM,\nu)$, the measure $\mu_{\bbM,\nu}$ on $\Str_L$ is an ergodic structure.
\end{lemma}

In fact, this AHK system is ``random-free''. This terminology comes from the world of graphons: a graphon is said to be \defn{random-free} \cite[\S10]{MR3043217} when it is $\{0,1\}$-valued almost everywhere. This can be thought of as ``having randomness'' only at the level of vertices (and not at higher levels --- namely edges, in the case of graphs). See also $0$--$1$ valued graphons in \cite{MR2815610} and the simple arrays of \cite{MR1702867}. A graphon is random-free if and only if the corresponding AHK system is random-free in the following sense.

\begin{definition}
An AHK system $(f_n)_{n\in \n}$ is \defn{random-free} if each function $f_n$ depends only on the singleton variables $\xi_{\{a\}}$ for $a\in \n$. An ergodic structure $\mu$ is \defn{random-free} if it has a random-free AHK representation.
\end{definition}

We would like to transfer properties of $\bbM$ to almost-sure properties of $\mu_{\bbM,\nu}$. It is not true in general that $\mu_{\bbM,\nu}\modelsas \Th_F(\bbM)$. But the following property will allow us to transfer satisfaction in $\bbM$ to satisfaction in $\mu_{\bbM,\nu}$ for $\Pi_2^-$ sentences.

\begin{definition}\label{def:strongwitnesses}
	Let $(\bbM,\nu)$ be a measured structure, and let $\varphi$ be a $\Pi_2^-$ sentence.
We say that \defn{$(\bbM, \nu)$ satisfies $\varphi$ with strong witnesses} (or \defn{has strong witnesses for $\varphi$}) if the following hold.
\begin{itemize}
\item If $\varphi$ is universal, then $\bbM\models \varphi$.
\item If $\varphi$ is pithy $\Pi_2$, i.e., of the form $\forall\xx\, \exists y\, \rho(\xx,y)$, then for every tuple $\aa$ from $\bbM$, the set $\rho(\aa,\bbM) = \{b\in \bbM\mid \rho(\aa,b)\}$ either contains an element of the tuple $\aa$ or has positive $\nu$-measure.
\end{itemize}
	For a $\Pi_2^-$ theory $T$,
we say that \defn{$(\bbM, \nu)$ satisfies $T$ with strong witnesses} when it satisfies $\varphi$ with strong witnesses for all $\varphi\in T$.
\end{definition}
Note that if $(\bbM, \nu)$ satisfies $T$ with strong witnesses, then $\bbM \models T$.

\begin{lemma}\label{lem:strongwitnesses}
Let $(\bbM,\nu)$ be a measured structure, and let $\mu = \mu_{\bbM,\nu}$. 
\begin{enumerate}[(i)]
\item Let $Q$ be a countable set of partial quantifier-free types. If $\bbM$ omits all the types in $Q$, then $\mu$ almost surely omits all the types in $Q$.
\item Let $T$ be a $\Pi_2^-$ theory. If $(\bbM, \nu)$ satisfies $T$ with strong witnesses, then $\mu$ almost surely satisfies $T$.
\item Further, if there is a quantifier-free formula $\chi(\xx)$ such that $\bbM$ is $\chi$-rooted with respect to quantifier-free types, then $\mu$ is properly ergodic.
\end{enumerate}
\end{lemma}
\begin{proof}
\emph{(i)}
If no tuple from $\bbM$ realizes a quantifier-free type $q\in Q$, then no tuple from any countable substructure sampled from $\bbM$ realizes $q$.
		
\vspace*{0.5em}
\noindent \emph{(ii)}
Every universal sentence $\forall\xx\, \psi(\xx)$ in $T$ is almost surely satisfied by $\mu$, since every tuple $\overline{v}$ from $\bbM$ satisfies the quantifier-free formula $\psi(\xx)$.

Next, consider sentences of the form $\forall\xx\,\exists y\, \rho(\xx,y)$. Fix an $n$-tuple $\aa$ from $\n$, where $n$ is the length of $\xx$. Corresponding to this tuple, we have a random tuple $\overline{v} \defas (v_{a_1},\dots,v_{a_n})$ sampled from $\bbM$. By the Fubini--Tonelli theorem, it suffices to show that for a measure one collection of values of this random tuple (e.g., those for which coordinates indexed by distinct natural numbers take distinct values), there is almost surely some $b\in \n$ such that $\bbM\models \rho\bigl(\overline{v},v_b \bigr)$.

By strong witnesses, $\rho(\overline{v},\bbM)$ either contains an element $v_{a_i}$ of the tuple $\overline{v}$ or has positive measure. In the first case, $v_{a_i}$ serves as our witness. In the second case, since there are infinitely many other independent random elements $(v_b)_{b\in \n\setminus\enum{\aa}}$, almost surely infinitely many of them land in the set $\rho(\cc,\bbM)$.

\vspace*{0.5em}
\noindent \emph{(iii)}
		By \emph{(ii)}, $\mu\modelsas T$, and since $\chi(\xx)\land (\bigwedge_{i\neq j} x_i\neq x_j)$ is consistent with $T$, $\mu(\chi(\xx))>0$. Let $p$ be any type containing $\chi(\xx)$, and let $q$ be its restriction to the quantifier-free formulas. To show that $\mu(p) = 0$, it suffices to show that $\mu(q) = 0$. 

Now since $\bbM$ is $\chi$-rooted with respect to quantifier-free types, $q$ has a root $v$ in $\bbM$. The probability that a tuple sampled from $\bbM$ satisfies $q$ is bounded above by the probability that the tuple contains the root $v$. This probability is $0$, since the measure $\nu$ is atomless. Hence, by Theorem~\ref{thm:characterization1}, $\mu$ is properly ergodic.\qedhere
\end{proof}

Thus, after applying the $\Pi_2^-$ transformation from \S\ref{sec:pithy} to an $F$-theory $T$, we have reduced the problem of constructing a properly ergodic structure almost surely satisfying $T$ to that of constructing a measured structure with the properties in Lemma~\ref{lem:strongwitnesses}. 

\begin{theorem}\label{thm:characterization2}
Let $F$ be a countable fragment of $\Lwow$, let $T$ be a complete $F$-theory with trivial dcl, let $\chi(\xx)$ be a formula in $F$, and let $M$ be a $\chi$-rooted model of $T$. Then there are continuum-many properly ergodic structures $\mu$ such that $\mu\modelsas T$.
\end{theorem}

\begin{proof}
We begin by applying Theorem~\ref{thm:pithy} to obtain a language $L'\supseteq L$, a $\Pi_2^-$ theory $T'$, and a countable set of partial quantifier-free types $Q$. Let $M'$ be the natural expansion of $M$ to an $L'$-structure. Then $M'$ is $R_\chi$-rooted, where $R_\chi(\xx)$ is the atomic $L'$-formula corresponding to the $L$-formula $\chi(\xx)$. By Corollary~\ref{cor:measurebijection}, it suffices to construct a properly ergodic $L'$-structure which almost surely satisfies $T'$ and omits the types in $Q$.

\bigskip

\noindent \emph{Part 1: The inverse system}

We construct a sequence $(A_k)_{k\in \n}$ of finite $L'$-structures, each of which is identified with a substructure of $M'$. Given a structure $A$, we define the structure $A^*$ to have underlying set $A\cup\{*\}$, where no new relations hold involving $*$. For each $k$, we equip the underlying set of each $A_k^*$ with a discrete probability measure $\nu_k$ that assigns positive measure to every element, and we fix a finite sublanguage $L_k$ of $L'$. 
Finally, we define connecting maps $g_k\colon A_{k+1}^*\to A_k^*$ such that $g_k(*) = *$ for all $k$, which preserve the measures and certain quantifier-free types, as follows:
\begin{enumerate}[(1)]
\item $\nu_{k+1}(g_k^{-1}[X]) = \nu_k(X)$ for all $X\subseteq A_k^*$.
\item If $\aa$ is a tuple of distinct elements from $A_{k+1}$ such that $g_k(\aa)$ is a tuple of distinct elements of $A_k$, then $\qftp_{L_k}(\aa) = \qftp_{L_k}(g_k(\aa))$. Note that we make no requirement if $g_k$ is not injective on $\enum{\aa}$ or if any element of $\aa$ is mapped to $*$.
\end{enumerate}

We enumerate all pithy $\Pi_2$ sentences
	in $T'$ as $\langle \varphi_k\rangle_{k\in\n}$ and the types in $Q$ as $\langle q_i\rangle_{k\in\n}$ with redundancies, so that each sentence and each type appears infinitely often in its list. We also enumerate the symbols in the language $L'$ as $\langle R_k\rangle_{k\in\n}$.

At stage $0$, we start with $A_0 = \emptyset$, the empty substructure of $M'$. Then $A_0^* = \{*\}$, and we set $\nu_0(\{*\}) = 1$ and $L_0 = \emptyset$. 

At stage $k+1$, we are given $A_k$, $\nu_k$, and $L_k$. We define $A_{k+1}$, $\nu_{k+1}$, $L_{k+1}$, and the connecting map $g_k$ in four steps.

\vspace*{0.5em}
\noindent \emph{Step 1:} Splitting the elements of $A_k$.

Enumerate the elements of $A_k$ as $\langle a_1,\dots,a_m\rangle$. We build intermediate substructures $B_i = \{a_1,\dots,a_m,a_1',\dots,a_i'\}$ of $M'$, where each new element $a_j'$ is a ``copy'' of $a_j$ to be defined. We start with $B_0 = A_k$. 

Given $B_i$, let $\varphi_{B_i}(x_1,\dots,x_m,x_1',\dots,x_i')$ be the conjunction of all atomic and negated atomic $L_k$ formulas holding on $B_i$, so that $\varphi_{B_i}$ encodes the quantifier-free $L_k$-type of $B_i$. Now there is an $L$-formula $\psi_{B_i}$ in $F$ such that $\psi_{B_i}$ has the same realizations as $\varphi_{B_i}$ in $M'$. Since $T = \Th_F(M)$ has trivial dcl, we can find another realization $a_{i+1}' \neq a_{i+1}$ of $\psi_{B_i}(a_1,\dots,x_{i+1},\dots,a_m,a_1',\dots,a_i')$ in $M'\setminus B_i$. Set $B_{i+1} = B_i\cup\{a_{i+1}'\}$. We have 
\[
	(\dagger) \hspace{.1in} \qftp_{L_k}(a_1,\dots,a_{i+1},\dots,a_m,a_1',\dots,a_i') = \qftp_{L_k}(a_1,\dots,a_{i+1}',\dots,a_m,a_1',\dots,a_i').
\]

At the end of Step 1, we have a structure $B_m = \{a_1,\dots,a_m,a_1',\dots,a_m'\}$.

\vspace*{0.5em}
\noindent \emph{Step 2:} Splitting $*$.

The pithy $\Pi_2$ sentence $\varphi_k$ has the form $\forall \xx\, \exists y\, \rho(\xx,y)$, where $\xx$ is a tuple of length $j$ and $\rho(\xx,y)$ is quantifier-free. Suppose there is a tuple $\aa$ from $B_m$ such that $B_m\models \lnot \exists y\, \rho(\aa,y)$. Then, since $M'\models \exists y\, \rho(\aa,y)$, we can choose some witness $c_{\aa}$ to the existential quantifier in $M'\setminus B_m$. Let $W = \{c_{\aa}\mid \aa\in B_m^j\text{~and~}B_m\models \lnot \exists y\,\rho(\aa,y)\}$ be the (finite) set of chosen witnesses. Note that if $\xx$ is the empty tuple of variables, then $W$ is either empty or consists of a single witness, depending on whether $B_m\models \exists y\, \rho(y)$. 

Let $A_{k+1} = B_m \cup W$ if $W$ is non-empty, and otherwise let $A_{k+1} = B_m \cup \{c\}$, where $c$ is any new element in $M'\setminus B_m$. 

\vspace*{0.5em}
\noindent \emph{Step 3:} Defining $g_k$ and $\nu_{k+1}$.

Recall that $g_k$ is to be a map from $A_{k+1}^*$ to $A_k^*$. We set $g_k(a_i) = g_k(a_i') = a_i$ and $g_k(c) = g_k(*) = *$ for $c\in A_{k+1}\setminus B_m$.

We define $\nu_{k+1}$ by splitting the measure of an element of $A_k^*$ evenly among its preimages under $g_k$. So $\nu_{k+1}(a_i) = \nu_{k+1}(a_i') = \frac{1}{2}\nu_k(a_i)$, and $\nu_{k+1}(c) = \nu_{k+1}(*) = \frac{1}{N}\nu_k(*)$, where $N = |A_{k+1}^*\setminus B_m|\geq 2$. Note that every element of $A_{k+1}^*$ has positive measure, by induction.

\vspace*{0.5em}
\noindent \emph{Step 4:} Defining $L_{k+1}$.

We expand the current language $L_k$ to $L_{k+1}$ by adding finitely many new symbols from $L'$.
\begin{enumerate}[(a)]
\item Add $R_{k}$ to $L_{k+1}$ if it is not already included. 
\item Since $A_{k+1}$ is a substructure of $M'$, no tuple from $A_{k+1}$ realizes $q_k$. That is, for every tuple $\aa$ from $A_{k+1}$, there is some quantifier-free formula $\varphi_{\aa}(\xx)\in q_k$ such that $M'\models \lnot \varphi_{\aa}(\aa)$. Add the finitely many relation symbols appearing in $\varphi_{\aa}$ to $L_{k+1}$.
\item Let $n$ be the number of free variables in $\chi(\xx)$. For every pair of $n$-tuples $\aa$ and $\bb$ from $A_{k+1}$ that realize distinct quantifier-free $L'$-types in $M'$, there is some relation symbol $R_{\aa,\bb}$ that separates their types. Add $R_{\aa,\bb}$ to $L_{k+1}$.
\end{enumerate}

This completes stage $k+1$ of the construction. Let us check that conditions (1) and (2) above are satisfied by the connecting map $g_k$.

\begin{enumerate}[(1):]
\item Since $\nu_k$ and $\nu_{k+1}$ are discrete measures on finite spaces, it suffices to check that $\nu_k(a) = \sum_{b\in g_k^{-1}[\{a\}]} \nu_{k+1}(b)$ for every singleton $a\in A_k^*$. This follows immediately from our definitions of $g_k$ and $\nu_{k+1}$.
\item Let $\bb$ be a tuple from $A_{k+1}$. The assumption that $g_k(\bb)$ is a tuple of distinct elements of $A_k$ means that every element of $\bb$ is in $B_m$ (since the other elements are mapped to $*$) and that $a_i$ and $a_i'$ are not both in $\bb$ for any $i$. For any function $\gamma\colon [m]\to [2]$, let $\aa^\gamma$ be the $m$-tuple which contains $a_i$ if $\gamma(i) = 0$ and $a_i'$ if $\gamma(i) = 1$. Then, expanding $\bb$ to an $m$-tuple of the form $\aa^\gamma$, it suffices to show that
$\qftp_{L_k}(\aa^\gamma) = \qftp_{L_k}(g_k(\aa^\gamma)) = \qftp_{L_k}(\aa)$. This follows by several applications of instances of the equality $(\dagger)$ above.
\end{enumerate}

\bigskip

\noindent \emph{Part 2: The measured structure}

Let $\mathbb{X}$ be the inverse limit of the system of sets $A_k^*$ and surjective connecting maps $g_k$. For each $k$, let $\pi_k$ be the projection map $\mathbb{X}\to A_k\cup\{*\}$. Then $\mathbb{X}$ is a profinite set, so it has a natural topological structure as a Stone space, in which the basic clopen sets are exactly the preimages under the maps $\pi_k$ of subsets of the sets $A_k^*$. Note that $\mathbb{X}$ is separable, so it is a standard Borel space.

Let $\nu^*$ be the finitely additive measure on the Boolean algebra $\mathcal{B}^*$ of clopen subsets of $\mathbb{X}$ defined by $\nu^*(\pi_k^{-1}[X]) = \nu_k(X)$. This is well defined by condition~(1). By the Hahn--Kolmogorov Measure Extension Theorem~\cite[Theorem 1.7.8]{Tao}, $\nu^*$ extends to a Borel probability measure $\nu$ on $\mathbb{X}$.

Now each element $a$ of $A_k^*$ has at least $2$ preimages in $A_{k+1}^*$, each of which have measure at most $\frac{1}{2}\nu_k(a)$. Hence, by induction, the measure of each element of $A_k^*$ is at most $2^{-k}$. So for all $x\in \mathbb{X}$, the point $x$ is contained in a basic clopen set $X_k = \pi_k^{-1}[\{\pi_k(x)\}]$ with $\nu(X_k)\leq 2^{-k}$ for all $k$. This implies that $\nu(\{x\}) = 0$ and $\nu$ is non-atomic.

Note that there is a unique element $*$ of $\mathbb{X}$ with the property that $\pi_k(*) = *$ for all $k$. We define a Borel $L'$-structure $\bbM$ with domain $\mathbb{X}\setminus \{*\}$ (which is also a standard Borel space). Since we have only removed a measure $0$ set from $\mathbb{X}$, the probability measure $\nu$ on $\mathbb{X}$ restricts to a probability measure on $\bbM$, which we also call $\nu$.

We define the structure on $\bbM$ by specifying the quantifier-free type of every tuple of distinct elements from $\bbM$. By Step 4~(a), $\bigcup_{k = 0}^\infty L = L'$. Given a tuple $\aa$ of distinct elements from $\bbM$ and a quantifier-free formula $\varphi(\xx)$, we choose $k$ large enough so that $L_k$ contains all of the relation symbols appearing in $\varphi(\xx)$ and so that $\pi_k(\aa)$ is a tuple of distinct elements from $A_k$. We set $\bbM\models \varphi(\aa)$ if and only if $A_k\models \varphi(\pi_k(\aa))$. This is well-defined by condition~(2).

By Step 4 (a), $\bigcup_{k = 0}^\infty L_k = L'$. Given a tuple $\aa$ from $\bbM$ and a symbol $R$ in $L'$, we choose $k$ large enough so that $R\in L_k$ and distinct elements of $\aa$ are mapped by $\pi_k$ to distinct elements of $A_k$. We set $\bbM\models R(\aa)$ if and only if $A_k\models R(\pi_k(\aa))$. This is well-defined by condition (2).

According to this definition, to determine whether a quantifier-free formula $\varphi(\xx)$ holds of a tuple $\aa$ with repeated elements, we can remove the redundancies from $\aa$ and replace the corresponding variables in $\xx$. For example, if $a_i = a_j$, we can remove $a_j$ and replace instances of $x_j$ in $\varphi(\xx)$ with $x_i$. This is equivalent to choosing $k$ large enough so that distinct elements of $\aa$ are mapped by $\pi_k$ to distinct elements of $A_k$ and checking whether $A_k\models \varphi(\pi_k(\aa))$. 

The interpretation of a relation symbol $R$ is then a Borel subset of $\bbM^{\ar(R)}$. Indeed, fixing $k$, the set of tuples $\aa$ such that distinct elements of $\aa$ are mapped by $\pi_k$ to distinct elements of $A_k$ and $\pi_k(\aa)$ satisfies $R$ is closed (the finite union of certain boxes intersected with certain diagonals), and the interpretation of $R$ is the countable union (over $k$) of these sets. Hence $\bbM$ is a Borel structure. 

We now verify the conditions of Lemma~\ref{lem:strongwitnesses} for the measured structure $(\bbM, \nu)$, the $\Pi_2^-$ theory $T'$, the quantifier-free types $Q$, and the quantifier-free formula $R_\chi(\xx)$.

\vspace{0.5em}
\noindent \emph{(i)}
$\bbM$ omits all the types in $Q$. 

Let $q(\xx)$ be a type in $Q$, and let $\aa$ be a tuple from $\bbM$. Let $k$ be large enough so that $\pi_k(\aa)$ is a tuple of distinct elements of $A_k$. Since $q$ appears infinitely many times in our enumeration of $Q$, there is some $l>k$ such that $q = q_l$. Then $\bb \defas \pi_{l+1}(\aa)$ is also a tuple of distinct elements of $A_{l+1}$. In Step 4 (b) of stage $l+1$ of the construction, we ensured that $L_{l+1}$ includes the relation symbols appearing in a quantifier-free formula $\varphi_{\bb}(\xx)\in q_k$ such that $A_{l+1}\models \lnot\varphi_{\bb}(\bb)$. Then also $\bbM\models \lnot \varphi_{\bb}(\overline{a})$, and hence $\aa$ does not realize $q$.

\vspace{0.5em}
\noindent	\emph{(ii)}
$(\bbM, \nu)$ satisfies $T'$ with strong witnesses.

Let $\varphi$ be a $\Pi_2^-$ sentence in $T'$. Then $\varphi$ has the form $\forall \xx\, \psi(\xx)$, where $\psi(\xx)$ is either quantifier-free or has a single existential quantifier. Let $\aa$ be a tuple from $\bbM$. Let $k$ be large enough so that all the symbols in $\varphi$ are in $L_k$ and $\pi_k(\aa)$ is a tuple of disjoint elements of $A_k$. 

If $\psi(\xx)$ is quantifier-free, then $\bbM\models \psi(\aa)$ if and only if $A_k\models \psi(\pi_k(\aa))$. The latter holds, since $A_k$ is a substructure of $\bbM$, and $\bbM\models \varphi$. 

Otherwise, $\psi(\xx)$ has the form $\exists y\, \rho(\xx,y)$, and since $\varphi$ appears infinitely many times in our enumeration of the pithy $\Pi_2$ sentences in $T'$, there is some $l>k$ such that $\varphi = \varphi_l$. Then $\pi_l(\aa)$ is a tuple of distinct elements of $A_l$, and $\bb \defas \pi_{l+1}(\aa)$ is a tuple of distinct elements of $A_{l+1}$. In Step 2 of stage $l+1$ of the construction, we ensured that there was some witness $c_{\bb}$ such that $A_{l+1}\models \rho(\bb,c_{\bb})$. If $c_{\bb}$ is not an element of the tuple $\bb$, then for any $c\in \bbM$ such that $\pi_{l+1}(c) = c_{\bb}$, we have $\bbM\models \rho(\aa,c)$. Since $\nu(\pi_{l+1}^{-1}[\{c\}]) = \nu_{l+1}(c) >0$, the set $\rho(\aa,\bbM)$ has positive $\nu$-measure. On the other hand, if $c_{\bb}$ is an element of the tuple $\bb$, say $b_i$, then $\bbM\models \rho(\aa,a_i)$. 

\vspace{0.5em}
\noindent	\emph{(iii)}
$\bbM$ is $R_\chi$-rooted with respect to quantifier-free types.

We would like to show that every non-redundant quantifier-free $n$-type containing $R_\chi(\xx)$ that is realized in $\bbM$ has a root in $\bbM$. Suppose not. Then there is a quantifier-free type $p(\xx)$ and a family of $n$-tuples $(\aa^i)_{i\in I}$ from $\bbM$ such that each $\aa^i$ realizes $p$, but there is no element $a$ which is in every $\aa^i$. Note that if such a family exists, then we can find one containing only finitely many tuples: picking some $\aa$ in the family, for each element $a_j$ in $\aa$ there is another tuple in the family which does not contain $a_j$, so $n+1$ tuples suffice.

Let $(\aa^1,\dots,\aa^m)$ be our finite family of tuples. Let $k$ be large enough so that $R_\chi\in L_k$ and $\pi_k$ is injective on $\bigcup_{i = 1}^m \enum{\aa^i}$. For all $i$, let $\bb^i = \pi_k(\aa^i)$. Then all of the tuples $\bb^i$ realize the same quantifier-free $L_k$-type $p' = p\restriction L_k$ in $A_k$, and $p'$ contains $R_\chi(\xx)$. By Step 4 (c) of stage $k$ of our construction, the tuples $\bb^i$ must actually realize the same quantifier-free $L'$ type $q\supseteq p'$ in $M'$ (which may be distinct from $p$). But there is no element which appears in all of these tuples, contradicting the fact that $M'$ is $R_\chi$-rooted.

\vspace{0.5em}
Let $\mu = \mu_{\bbM,\nu}$. By Lemma~\ref{lem:strongwitnesses}, $\mu$ is a properly ergodic structure that almost surely satisfies $T'$ and omits the types in $Q$.

\bigskip

\noindent \emph{Part 3: Rescaling to obtain continuum-many properly ergodic structures}

Again, let $n$ be the number of free variables in $\chi(\xx)$. For any quantifier-free formula $\varphi(x_1,\dots,x_n)$, we define the quantifier-free formula $\varphi^*(x_1,\dots,x_n)$: $$\bigvee_{\sigma\in \sympar{n}} \varphi(x_{\sigma(1)},\dots,x_{\sigma(n)}).$$ Note that the set $\varphi^*(\bbM) = \{\aa\in \bbM^n\mid \bbM\models \varphi^*(\aa)\}$ is invariant under the natural action of the symmetric group $\sympar{n}$ on $\bbM^n$ by permuting coordinates.

We will use the following claim to apply the rescaling technique from~\cite{AFKP}.

\vspace*{0.5em}

\noindent {\bf Claim:} There is some quantifier-free formula $\varphi(\xx)$ such that $0 < \nu^n(\varphi^*(\bbM)) < 1$.

\noindent \emph{Proof of Claim.}
Suppose not. Then for every quantifier-free formula $\varphi(\xx)$, $\nu^n(\varphi^*(\bbM))$ is equal to $0$ or $1$. Note that $\nu^n(\varphi^*(\bbM))>0$ if and only if $\nu^n(\varphi(\bbM))>0$. In particular, since $\nu^n(R_\chi(\bbM))>0$, we have $\nu^n(R_\chi^*(\bbM)) = 1$.

Let $A\subseteq \bbM^n$ be the set of tuples satisfying the partial quantifier-free type $\{\varphi^*(\xx)\mid \nu^n(\varphi^*(\bbM)) = 1\} \cup \{\neg\varphi^*(\xx)\mid \nu^n(\varphi^*(\bbM)) = 0\}$. Since $A$ is a countable intersection of measure $1$ sets, it has measure $1$. Pick a tuple $\aa\in A$. Some permutation of $\aa$ satisfies $R_\chi$, and $A$ is $\sympar{n}$-invariant, so we may assume that $\bbM\models R_\chi(\aa)$. Let $p(\xx) = \qftp(\aa)$. Since $R_\chi(\xx)\in p(\xx)$, some coordinate $a_i$ of $\aa$ is a root for $p(\xx)$. 

Now $\nu$ is atomless, so the set of tuples in $\bbM^n$ containing $a_i$ has measure $0$. Thus we can pick another tuple $\bb\in A$ which does not contain $a_i$. By rootedness, no permutation of $\bb$ satisfies $p(\xx)$.

In particular, there is some quantifier-free formula $\psi(\xx)\in p(\xx)$ such that no permutation of $\bb$ satisfies $\psi(\xx)$ (explicitly, take the conjunction of $n!$ formulas in $p(\xx)$, one separating $p(\xx)$ from $\qftp(\sigma(\bb))$ for each $\sigma\in \sympar{n}$). We therefore have $\bbM\models \lnot \psi^*(\bb)$. But we also have that $\psi^*(\xx) \in p(\xx)$ and so $\bbM\models \psi^*(\aa)$. Hence every tuple in $A$ must satisfy $\psi^*(\xx)$, contradicting the fact that $\bb\in A$.
\qed
\vspace*{0.5em}

In~\cite{AFKP}, a method is described for rescaling a probability measure $\mu$ according to a weight $\mathcal{W}$ (essentially an assignment of weights to the pieces of a finite partition of the domain) to obtain a new probability measure $\mu^{\mathcal{W}}$. In that paper, all probability measures are continuous measures on $\mathbb{R}$, but the results apply equally well to measures on $\bbM$, since this is a standard Borel space. 

The main observation about this construction is that $\mu$ and $\mu^{\mathcal{W}}$ are equivalent measures, in the sense that they are absolutely continuous with respect to each other. It follows that for our measure $\nu$ on $\bbM$, any measure of the form $\nu^{\mathcal{W}}$ is an atomless probability measure on $\bbM$, with the property that $(\bbM,\nu^{\mathcal{W}})$ satisfies $T'$ with strong witnesses, and hence the measure $\mu_\mathcal{W} = \mu_{\bbM,\nu^{\mathcal{W}}}$ on $\Str_{L'}$ is a properly ergodic structure which almost surely satisfies $T'$ and omits the types in $Q$.

Now by the key proposition~\cite[Proposition 3.8]{AFKP}, since the set $\varphi^*(\bbM)$ from the Claim above is an $\sympar{n}$-invariant Borel set with $\nu^n$-measure strictly between $0$ and $1$, the expression $\nu^{\mathcal{W}}(\varphi^*(\bbM))$ takes on continuum-many values as $\mathcal{W}$ varies through the possible weights. And since $\mu_{\mathcal{W}}(\extent{\varphi^*(\aa)}) = \nu^{\mathcal{W}}(\varphi^*(\bbM))$ for any tuple $\aa$ of distinct elements of $\mathbb{N}$, this construction produces continuum-many properly ergodic structures of the form $\mu_{\mathcal{W}}$. 
\end{proof}

Theorem~\ref{thm:characterization2}, along with the results of the previous sections, gives a ``measure-free'' characterization of those theories which admit properly ergodic models.

\begin{theorem}\label{thm:main}
Suppose $\Sigma$ is a set of sentences in some countable fragment $F$ of $\Lwow$. The following are equivalent:
\begin{enumerate}
\item There is a properly ergodic structure $\mu$ such that $\mu\modelsas \Sigma$.
\item There are continuum-many properly ergodic structures $\mu$ such that $\mu\modelsas \Sigma$. 
\item There is a countable fragment $F'\supseteq F$ of $\Lwow$, a complete $F'$-theory $T\supseteq \Sigma$ with trivial dcl, a formula $\chi(\xx)$ in $F'$, and a model $M\models T$ which is $\chi$-rooted.
\end{enumerate}
\end{theorem}
\begin{proof}
$(3)\rightarrow (2)$: By Theorem~\ref{thm:characterization2}, there are continuum-many properly ergodic structures $\mu$ such that $\mu\modelsas T$, and $\Sigma\subseteq T$.

$(2)\rightarrow (1)$: Clear.

$(1)\rightarrow (3)$: Theorem~\ref{thm:characterization1} gives us a countable fragment $F'\supseteq F$, and a formula $\chi(\xx)$ in $F'$ such that $\mu(\chi(\xx))>0$, but for every $F'$-type $p$ containing $\chi(\xx)$, $\mu(p) = 0$. Let $T = \Th_F(\mu)$. Then $\Sigma\subseteq T$, and $T$ has trivial dcl by Theorem~\ref{thm:trivialdcl}. Now by Theorem~\ref{thm:rootedness}, the set of $\chi$-rooted models of $T$ has measure $1$. In particular, it is non-empty.
\end{proof}

\begin{remark}\label{rem:infinitarynecessity}
The conditions in Theorem~\ref{thm:main} (3) can sometimes be satisfied with $F' = F$. In fact, for many of the examples in Section~\ref{sec:examples}, we could take $F'$ to be first-order logic. However, Example~\ref{ex:infinitary} shows that, in general, the move to a larger fragment of $\Lwow$ is necessary.
\end{remark}

The following two corollaries, which may be of interest independently of Theorem~\ref{thm:main}, follow immediately from its proof in the case that $\mu$ is properly ergodic and from the analogous construction in \cite{AFP} in the case that $\mu$ is almost surely isomorphic to a countable structure.

\begin{corollary}\label{cor:borelmodel}
If $\mu$ is an ergodic structure, then for any countable fragment $F$ of $\Lwow$, the theory $\Th_F(\mu)$ has a Borel model (of cardinality $2^{\aleph_0}$).
\end{corollary}

\begin{corollary}\label{cor:randomfree}
For every countable fragment $F$ of $\Lwow$, every ergodic structure $\mu$ is $F$-elementarily equivalent to a random-free ergodic structure $\mu'$. That is, there exists a random-free ergodic structure $\mu'$ such that $\Th_F(\mu) = \Th_F(\mu')$. 

Moreover, except in the case that $\mu$ concentrates on the isomorphism type of a highly homogeneous structure $M$, there exist continuum-many such $\mu'$.
\end{corollary}

For a definition and discussion of high homogeneity, and the previously-mentioned characterization \cite{MR0401885} of highly homogeneous structures as those interdefinable with one of the five reducts of the rational linear order, see \cite[\S2.3]{AFKP}.

\section*{Acknowledgments}
Portions of this work appear in A.K.'s PhD thesis \cite[Chapters~1 and~2]{kruckman-thesis}, completed under the supervision of Thomas Scanlon, whom we thank for his support and insightful comments. We also thank Will Boney, Alexander Kechris, and Slawomir Solecki for helpful discussions, and Andrew Marks for bringing Example~\ref{ex:geometric} to our attention.

This research was facilitated by the Mathematical Sciences Research Institute program on Model Theory, Arithmetic Geometry, and Number Theory (January--May 2014), the workshop at U.C.\ Berkeley on Vaught's Conjecture (June 2015), and the Lorentz Center workshop on Logic and Random Graphs (August--September 2015).

Work on this publication by C.F.\ was made possible through the support of ARO grant W911NF-13-1-0212 and a grant from Google.


\newcommand{\etalchar}[1]{$^{#1}$}
\providecommand{\bysame}{\leavevmode\hbox to3em{\hrulefill}\thinspace}
\providecommand{\MR}{\relax\ifhmode\unskip\space\fi MR }
\providecommand{\MRhref}[2]{%
  \href{http://www.ams.org/mathscinet-getitem?mr=#1}{#2}
}
\providecommand{\href}[2]{#2}

\end{document}